%% file: main_arxiv.tex
\documentclass{article}
\usepackage[margin=1in]{geometry}
\usepackage[T1]{fontenc}
\usepackage{graphicx}
\usepackage{amsmath,amssymb,amsfonts,amsthm}
\input{math_commands.tex}
\usepackage{booktabs}
\usepackage{xcolor}
\usepackage[colorlinks=true, urlcolor=blue, citecolor=blue, linkcolor=blue]{hyperref}
\usepackage{algorithm, algpseudocode}

\newtheorem{theorem}{Theorem}
\newtheorem{proposition}{Proposition}
\newtheorem{lemma}{Lemma}
\newtheorem{corollary}{Corollary}
\theoremstyle{remark}
\newtheorem{remark}{Remark}

\renewenvironment{proof}[1][Proof]{\par\noindent\textit{#1.}\ }{\par}

\newcommand{\email}[1]{\href{mailto:#1}{\nolinkurl{#1}}}
\newcommand{\keywords}[1]{%
  \par\medskip\noindent\textbf{Keywords: }%
  \begingroup\def\and{; }#1\endgroup}
\newcommand{\subclass}[1]{%
  \par\smallskip\noindent\textbf{Mathematics Subject Classification: }%
  \begingroup\def\and{, }#1\endgroup}

\begin{document}

\title{Coordinate Optimality Reformulation for Mixed-Integer Convex Programs with Indicators
\thanks{Simge K\"u\c{c}\"ukyavuz and Tong Xu are supported, in part, by ONR N00014-22-1-2602. Salar Fattahi is supported, in part, by NSF grants 2152776 and 2337776, and ONR grant N00014-26-1-2074. Andr\'es G\'omez is supported in part by ONR grant N000142612117 and AFOSR grant FA9550-24-1-0086.}
}

\author{
Tong Xu\thanks{Department of Industrial Engineering and Management Sciences, Northwestern University, Evanston, IL, USA. \email{tongxu2027@u.northwestern.edu}}
\and
Salar Fattahi\thanks{Department of Industrial and Operations Engineering, University of Michigan, Ann Arbor, MI, USA. \email{fattahi@umich.edu}}
\and
Andr{\'e}s G{\'o}mez\thanks{Daniel J. Epstein Department of Industrial and Systems Engineering, University of Southern California, Los Angeles, CA, USA. \email{gomezand@usc.edu}}
\and
Simge K\"u\c{c}\"ukyavuz\thanks{Department of Industrial Engineering and Management Sciences, Northwestern University, Evanston, IL, USA. \email{simge@northwestern.edu}}
}

\date{}

\maketitle

\begin{abstract}
We consider mixed-integer convex optimization problems in which binary indicators control continuous variables. We introduce the \emph{Coordinate Optimality Reformulation} (CORe) framework, which augments standard indicator formulations by incorporating coordinate-wise optimality information. The resulting reformulations preserve global optimality while substantially improving branch-and-bound performance, particularly in sparse and structured settings where the coordinate-wise optimality conditions expose exploitable problem structure. We first develop the main components of CORe, including coordinate-wise optimality conditions, closed-form characterizations, and disjunctive reformulations. We then demonstrate the framework across multiple problem families, including quadratic problems and robust single-index models. Computational experiments show that CORe can substantially improve solver performance compared with standard big-$M$ formulations. 

\subclass{90C11 \and 90C20}
\end{abstract}

\section{Introduction}
\label{intro}

We consider a class of mixed-integer convex optimization problems in which binary indicators control the activation of a subset of continuous variables. A canonical formulation takes the form
\begin{equation}
\label{eq:core_template}
\min_{x \in \mathbb{R}^n,\, y \in \mathbb{R}^m,\, z \in \{0,1\}^n}
\quad f(x,y) + \sum_{i=1}^n \lambda_i z_i
\quad \text{s.t.} \quad x_i (1 - z_i) = 0,\; i = 1,\dots,n, 
\end{equation}
where $f$ is convex. The binary variables $z_i, i = 1,\dots,n$ indicate whether the corresponding continuous variables $x_i, i = 1,\dots,n$ are active. When present (i.e., $m\geq 1$), $y$ denotes a collection of continuous variables that are not directly controlled by the indicators, and $\lambda\in\mathbb{R}_+^n$ is the given penalty vector. Throughout, we refer to the constraints $x_i(1 - z_i) = 0$ as indicator constraints. The indicator constraints can be linearly represented by $-M_iz_i \leq x_i\leq M_i z_i$, where $M_i$ is a sufficiently large constant (commonly referred to as a big-M constant) \cite{glover1975improved}.

Quadratic indicator models provide a particularly important special case of
\eqref{eq:core_template} and will serve as one of the main settings in this paper. In this case, $m=0$, and the model reduces to a convex quadratic program with indicators of the form
\begin{equation}
\label{eq:original_core}
\begin{aligned}
\min_{x\in\mathbb{R}^n,\; z\in\{0,1\}^n}\quad
& \frac12 x^\top Qx + c^\top x + \lambda^\top z \\
\text{s.t.}\quad
& -M_iz_i \leq x_i\leq M_i z_i, i=1,\ldots,n,
\end{aligned}
\end{equation}
where \(Q\succeq 0\) is the quadratic cost matrix and \(c\in\mathbb{R}^n\). Despite the convexity of the quadratic objective, the presence of indicator constraints renders the problem NP-hard \cite{chen2014complexity}.

The optimization model \eqref{eq:core_template} encompasses many problems, including learning probabilistic graphical models \cite{manzour2021integer,kucukyavuz2023consistent,JMLR:v26:24-1657,xu2025integer,fattahi2023solution,fattahi2021scalable}, gene regulatory networks \cite{ravikumar2025efficient}, portfolio optimization \cite{bienstock1996computational}, unit commitment \cite{frangioni2006perspective,nguyen2022mixed}, manufacturing \cite{akturk2009strong}, outlier detection \cite{zioutas2005deleting,gomez2021outlier,gomez2025outlier}, sparse regression \cite{bertsimas2016best,mazumder2023subset,dedieu2021learning}, and robust regression \cite{bertsimas2016best,hazimeh2022sparse,gomez2025outlier}. Big-M-constrained formulations are a common starting point for these applications. While broadly applicable, such formulations often provide limited information to the branch-and-bound algorithm: continuous relaxations may be weak, presolve may fail to identify the relevant activation regimes, and the solver may need to explore many suboptimal activation patterns. 

Existing approaches often resort to either strengthening continuous relaxation or developing specialized algorithms that exploit the structural properties of the problem. A major line of work strengthens indicator models through perspective reformulations and related convexification techniques \cite{ceria1999convex,frangioni2006perspective,gunluk2010perspective,atamturk2018strong,wei2022ideal}. Other approaches exploit additional structure in sparse learning and quadratic optimization with indicators, including cutting-plane and first-order methods, low-rank structure, and graph-induced structure in the quadratic matrix \cite{bertsimas2016best,hazimeh2020fast,mazumder2023subset,liu2023graph,bhathena2025parametric,bienstock2024solving}. While these methods can be highly effective within their target settings, they are tailored to specific problem structures. 

\begin{sloppypar}
In this paper, we introduce the \emph{Coordinate Optimality Reformulation} (CORe) framework, which provides a general reformulation principle based on coordinate-wise optimality that can be specialized to a broad range of indicator models, while still delivering partial benefits even in the absence of favorable structures. The key observation is that, under the canonical formulation \eqref{eq:core_template}, each pair $(x_i,z_i)$ defines a one-dimensional coordinate subproblem when the remaining variables are fixed. When this coordinate subproblem admits a tractable optimality characterization, the active and inactive regimes can be compared explicitly, resulting in additional constraints that can only strengthen the formulation. CORe systematically incorporates these comparisons into the mixed-integer formulation, yielding reformulations that preserve global optimality while eliminating activation patterns that cannot occur at any optimum. In doing so, CORe exposes structure that is invisible in the original formulation and can be exploited by presolve, branching, and cutting-plane routines.
\end{sloppypar}

\subsection{Contributions}

The main contributions of the paper are summarized below.
\begin{itemize}
    \item We introduce the \emph{Coordinate Optimality Reformulation} (CORe) framework for a class of mixed-integer convex optimization problems with indicator constraints. The framework has three components: an \emph{optimality-condition (OC) reformulation}, a \emph{closed-form (CF) reformulation}, and a \emph{disjunctive reformulation}. Together, these components incorporate coordinate-wise optimality conditions directly into the mixed-integer formulation in the form of constraints and make explicit a structure that standard formulations typically leave implicit in the continuous subproblem.

    \item We show that our framework applies across several structured model classes. In particular, for quadratic programs with indicators, the coordinate-wise comparison yields explicit threshold rules and disjunctive formulations. We further show that the same principle extends beyond quadratic objectives by applying it to robust single-index models---a class of problems in robust statistics. While the CORe conditions remain valid beyond these structured settings, the strongest computational benefits are expected when favorable underlying problem structures are present.

    \item We characterize special cases in which CORe sharply reduces the number of activation patterns that need to be considered. In convex quadratic problems whose sparsity graph is a star, the regimes of all leaf coordinates are controlled by the single continuous variable at the center, which leads to a linear-size bound on the number of patterns and an $O(n)$ bound on the size of the branch-and-bound tree. More generally, in quadratic problems with an unbalanced bipartite structure, in which the $n$ indicator-controlled variables interact only through $m$ additional continuous variables, the activation decisions are governed by a hyperplane arrangement in dimension $m$, yielding a polynomial $O(n^{m+1})$ bound on the number of branch-and-bound nodes for fixed $m$. An analogous $O(n^{d+1})$ bound holds for robust single-index models with $d$ covariates, a class that includes penalized least trimmed squares and robust logistic regression. In all these cases, the standard big-$M$ formulation can require exponentially many nodes. We further show that such guarantees do not follow from sparsity alone; even with CORe, path-structured instances can induce exponentially many branch-and-bound nodes.
    \item We show computationally that CORe can substantially outperform standard big-$M$ formulations and, in some settings, even specialized dynamic programming algorithms that exploit problem structure.
\end{itemize}

\subsection{Related Work}

\begin{sloppypar}
    \paragraph{General mixed-integer convex optimization.}
Indicator-constrained convex optimization is part of the broader literature on mixed-integer convex and nonlinear optimization. Classical algorithmic approaches include outer approximation and branch-and-bound for convex mixed-integer nonlinear programs (MINLPs) \cite{duran1986outer,quesada1992lp}, as well as hybrid branch-and-bound, branch-and-cut, and outer-approximation frameworks implemented in modern convex MINLP solvers \cite{bonami2008algorithmic,bonami2011algorithms}. Related mixed-integer conic and convex approaches strengthen relaxations through lift-and-project cuts, conic certificates, and extended polyhedral approximations \cite{kilincc2017lift,lubin2018polyhedral,coey2020outer,kilincc2016minimal,kilinckarzan2015twoterm}. These approaches typically reformulate or relax the original feasible set without excluding feasible integer solutions. CORe takes a different route: it adds constraints derived from coordinate-wise comparisons between active and inactive regimes. These constraints may remove feasible but suboptimal activation patterns, while preserving all globally optimal solutions.
\end{sloppypar}

\paragraph{Formulation strengthening for indicator variables.}
Perspective reformulations provide a powerful approach for mixed-integer convex problems with indicator variables.
In diagonal quadratic indicator models, perspective reformulation yields the ideal convex-hull description of the corresponding separable on-off quadratic epigraphs \cite{ceria1999convex,gunluk2010perspective}. The perspective strengthening approach has been extended to settings with nonseparable quadratic objectives, coupling constraints, and structured decompositions \cite{frangioni2006perspective,akturk2009strong,gunluk2010perspective,wei2020convexification,xie2020scalable,wei2022ideal}.  For nonseparable objectives, perspective-based approaches typically decompose or lift the quadratic terms so that separable on-off components can be convexified; the resulting relaxation depends on the chosen decomposition or extended formulation \cite{frangioni2007sdp,zheng2014improving,atamturk2018strong,frangioni2020decompositions}. Conic optimization perspectives have also been used to relate regularization and relaxation approaches in statistical variable selection \cite{dong2015regularization}. Other related developments include projected and approximated perspective reformulations, project-and-lift approaches, and dual-information strengthening for semi-continuous and nonseparable models \cite{frangioni2011projected,frangioni2016approximated,frangioni2017improving}.

\paragraph{Algorithms for sparse statistical learning.}
A separate line of work develops scalable algorithms for sparse statistical learning with indicator variables, including mixed-integer formulations for best subset selection and binary convex reformulations with cutting-plane algorithms for high-dimensional sparse regression \cite{bertsimas2016best,bertsimas2020sparseHD}. Complementary first-order and local-search approaches have been developed for $\ell_0$-regularized regression and classification, including coordinate descent, local combinatorial optimization, and specialized branch-and-bound methods whose node relaxations are solved by first-order algorithms \cite{hazimeh2020fast,hazimeh2022sparse,dedieu2021learning}. These methods exploit the structure of sparse statistical learning problems through specialized search, coordinate updates, or tailored relaxations. CORe uses a different mechanism: it extracts coordinate-wise information from the continuous subproblem and encodes it as constraints in a mixed-integer formulation.

\paragraph{Structure-exploiting methods for quadratic indicator problems.}
Another line of work studies convex quadratic indicator-constrained problems with special structure in the matrix \(Q\). The simplest case is diagonal \(Q\), where the continuous variables decouple and perspective reformulation yields the ideal convex-hull description of the feasible set \cite{ceria1999convex,gunluk2010perspective}. Beyond this fully separable setting, several papers identify classes of matrices for which strong structural results are available. When \(Q\) is Stieltjes or closely related to an \(M\)-matrix, the sign structure of the quadratic couplings can be exploited to derive strong convex formulations and, under additional assumptions, polynomial-time solvability results \cite{atamturk2018strong,liu2025polyhedral}. Rank-one quadratic structure has also been exploited to derive strong lifted relaxations for sparse regression \cite{atamturk2025rank}. Related pairwise convexification approaches study \(2\times 2\) quadratic substructures with indicators \cite{han20232}. Likewise, if \(Q\) admits a sparse factorization \(Q = Q_0^\top Q_0\), polynomial-time solvability can be obtained under appropriate assumptions on the factor structure \cite{del2020subset}. 

A particularly active line of research exploits the sparsity graph of \(Q\), i.e., the graph whose adjacency matrix is given by $Q$. When the sparsity graph of $Q$ is a star, i.e., there exists a center node $r \in [n]:=\{1,\ldots,n\}$ such that $Q_{ij} = 0$ for all $i, j \in [n] \setminus \{r\}$ with $i \neq j$, then the resulting problem is known to be polynomial-time solvable following a parametric method \cite{bhathena2025parametric}. When \(Q\) is tridiagonal, the associated sparsity graph is a path, and dynamic-programming methods can solve the problem to optimality in \(O(n^2)\) time and memory \cite{liu2023graph}. For banded matrices, decision-diagram and approximation approaches exploit temporal or local dependence to obtain tractable algorithms under fixed structural parameters \cite{gomez2024real}. 
When the sparsity graph of $Q$ has a tree structure, exact parametric dynamic programming algorithms solve the problem in \(O(n^2)\) time and memory \cite{bhathena2025parametric}. More recent work extends this parametric viewpoint to structured graphs whose complexity depends on treewidth, volume growth, and related graph parameters \cite{bhathena2026solvingconvexquadraticoptimization}. Related work on structured sparsity under indicator conditions develops approximation algorithms and complexity results for broader block-structured quadratic settings \cite{bienstock2024solving}. More broadly, sparsity and decomposition ideas based on graph structure have played an important role in polynomial and semidefinite optimization \cite{waki2006sums,madani2017finding,bienstock2018lp}.

Taken together, this literature shows that strong computational guarantees are achievable when the problem has a highly exploitable structure. These methods are inherently ad hoc, each tailored to a specific problem class (e.g., quadratic programs from a sparsity graph family). CORe takes a different approach: rather than exploiting global structure, it derives optimality implications from one-dimensional coordinate subproblems, yielding a unified reformulation principle that specializes naturally across model classes.

\paragraph{Disjunctive and geometric formulations.}
CORe also has a geometric and disjunctive interpretation: coordinate optimality conditions partition the continuous space into regions associated with activation regimes. This relates to classical disjunctive programming and convex-hull formulations \cite{balas1998disjunctive,grossmann2003generalized}, as well as to strong formulations for disjunctive constraints and unions of polyhedra \cite{vielma2015formulation,huchette2019combinatorial} and geometric constructions for strong MIP formulations \cite{huchette2019geometric}. While these works are not directly concerned with coordinate optimality, they offer a useful perspective on the regime partitions induced by CORe.

In summary, the related literature provides algorithmic and modeling tools for specific classes of indicator problems, while CORe contributes a unifying reformulation principle grounded in coordinate-wise optimality.

\subsection{Outline}

The remainder of the paper is organized as follows. Section~\ref{sec:cor_framework} introduces the general CORe framework and formalizes the coordinate-optimality logic underlying our approach. Section~\ref{sec:applications_QP} develops CORe for quadratic programs with indicators, including general sparse quadratic models and models with unbalanced bipartite structure. Section~\ref{sec:single_index} extends the same perspective to a class of robust single-index models, which includes penalized least trimmed squares and robust logistic regression as special cases, and establishes a polynomial bound on the size of the branch-and-bound tree when the number of covariates is fixed. Section~\ref{sec:experiments} presents computational experiments. We conclude in Section~\ref{sec:conclusion}.

\section{The Coordinate Optimality Reformulation (CORe) Framework}
\label{sec:cor_framework}

Recall the mixed-integer convex optimization problem
\begin{equation*}
\begin{aligned}
\min_{x \in \mathbb{R}^n,\; y \in \mathbb{R}^m,\; z \in \{0,1\}^n}
\quad
& f(x,y) + \sum_{i=1}^n \lambda_i z_i \\
\text{s.t.}\quad
& x_i(1-z_i)=0,\qquad i=1,\ldots,n.
\end{aligned}
\end{equation*}
For a vector \(x\in\mathbb{R}^n\) and an index \(i\), let \(x_{-i}\) denote all
coordinates of \(x\) except \(x_i\). For fixed \((x_{-i},y)\), define
\[
f_i(\xi;x_{-i},y)
:=
f(x_1,\ldots,x_{i-1},\xi,x_{i+1},\ldots,x_n,y).
\]
The coordinate-restricted subproblem in \((x_i,z_i)\) is thus defined as
\begin{equation}\label{eq::coordinate-subprob}
    \min_{x_i\in\mathbb{R},\; z_i\in\{0,1\}}
\quad
f_i(x_i;x_{-i},y)+\lambda_i z_i
\quad
\text{s.t.}\quad
x_i(1-z_i)=0.
\end{equation}
The optimal solution to this problem can be determined by comparing two regimes:

\medskip
\noindent
\textbf{Off-regime \((z_i=0)\).}
Feasibility forces \(\xi=0\), giving the value
\[
V_i^{\mathrm{off}}(x_{-i},y)
:=
f_i(0;x_{-i},y).
\]

\medskip
\noindent
\textbf{On-regime \((z_i=1)\).}
The coordinate is free, giving the value
\[
V_i^{\mathrm{on}}(x_{-i},y)
:=
\min_{x_i\in\mathbb{R}} f_i(x_i;x_{-i},y)+\lambda_i .
\]
When the infimum is attained, we denote the set of coordinate minimizers by
\[
X_i^\star(x_{-i},y)
:=
\arg\min_{x_i\in\mathbb{R}} f_i(x_i;x_{-i},y).
\]

\begin{proposition}[Coordinate optimality principle]
\label{prop:cor_principle}
Let \((x^\star,y^\star,z^\star)\) be a globally optimal solution to
\eqref{eq:core_template}. Then, for every coordinate \(i\),
\begin{equation*}
    \begin{cases}
    z_i^\star=0
& \quad\Longrightarrow\quad x_i^\star=0
\quad\text{and}\quad
V_i^{\mathrm{off}}(x_{-i}^\star,y^\star)
\le
V_i^{\mathrm{on}}(x_{-i}^\star,y^\star),\\
z_i^\star=1
& \quad\Longrightarrow\quad
x_i^\star\in X_i^\star(x_{-i}^\star,y^\star)
\quad\text{and}\quad
V_i^{\mathrm{on}}(x_{-i}^\star,y^\star)
\le
V_i^{\mathrm{off}}(x_{-i}^\star,y^\star).
\end{cases}
\end{equation*}

\end{proposition}

\begin{proof}
Fix a coordinate \(i\) and keep \((x_{-i}^\star,y^\star)\) fixed. If
\(z_i^\star=0\), then feasibility implies \(x_i^\star=0\). If
\[
V_i^{\mathrm{on}}(x_{-i}^\star,y^\star)
<
V_i^{\mathrm{off}}(x_{-i}^\star,y^\star),
\]
then switching only the \(i\)th coordinate to an optimal on-regime solution would
strictly decrease the objective while preserving feasibility, contradicting global
optimality. Therefore
\[
V_i^{\mathrm{off}}(x_{-i}^\star,y^\star)
\le
V_i^{\mathrm{on}}(x_{-i}^\star,y^\star).
\]

If \(z_i^\star=1\), then \(x_i^\star\) must minimize
\(f_i(\cdot;x_{-i}^\star,y^\star)\). Otherwise, replacing \(x_i^\star\) by a
coordinate minimizer would strictly decrease the objective while keeping the same
value of \(z_i^\star\), again contradicting global optimality. Hence
\(x_i^\star\in X_i^\star(x_{-i}^\star,y^\star)\). Finally, if
\[
V_i^{\mathrm{off}}(x_{-i}^\star,y^\star)
<
V_i^{\mathrm{on}}(x_{-i}^\star,y^\star),
\]
then switching to the off-regime \(x_i=0,z_i=0\) would strictly improve the
objective, which is impossible at a global optimum. This proves the result.\qed
\end{proof}

\begin{sloppypar}
    Proposition~\ref{prop:cor_principle} shows that every globally optimal solution
satisfies a coordinate-wise consistency condition. CORe incorporates these necessary optimality conditions into the mixed-integer formulation. The resulting constraints may be implicit at this level of generality, but in the structured models considered later, they reduce to explicit linear constraints, closed-form coordinate updates, or finite disjunctions.
\end{sloppypar}

\subsection{The Optimality Condition (OC) Component}

The OC component of CORe uses the optimality condition
of the coordinate subproblem \eqref{eq::coordinate-subprob}. For each coordinate \(i\), define the off-regime set
\[
\begin{aligned}
\mathcal D_{i, \mathrm{OC}}^{0}
:=
\Bigl\{(x,y,z):\;& z_i=0,\; x_i=0,\\
& f_i(0;x_{-i},y)
\le
f_i(\xi;x_{-i},y)+\lambda_i
\quad \forall \xi\in\mathbb{R}
\Bigr\},
\end{aligned}
\]
and the on-regime set
\[
\begin{aligned}
\mathcal D_{i, \mathrm{OC}}^{1}
:=
\Bigl\{(x,y,z):\;& z_i=1,\;
0\in \partial_{x_i} f_i(x_i;x_{-i},y),\\
& f_i(x_i;x_{-i},y)+\lambda_i
\le
f_i(0;x_{-i},y)
\Bigr\}.
\end{aligned}
\]
The off-regime set enforces that \(x_i=0\) and that switching coordinate \(i\) on
cannot improve the objective. The on-regime set enforces coordinate-wise stationarity
and that the on-regime is no worse than the off-regime.

The OC component of CORe is therefore the disjunctive constraint
\[
(x,y,z)\in
\mathcal D_{i, \mathrm{OC}}^{0}\cup \mathcal D_{i, \mathrm{OC}}^{1},
\qquad i=1,\ldots,n.
\]
This yields the following CORe formulation:
\begin{equation}
\label{eq:core_oc_formulation}
\begin{aligned}
\min_{x,y,z}\quad
& f(x,y)+\sum_{i=1}^n \lambda_i z_i\\
\text{s.t.}\quad
& (x,y,z)\in
\mathcal D_{i, \mathrm{OC}}^{0}\cup \mathcal D_{i, \mathrm{OC}}^{1},
\qquad i=1,\ldots,n,\\
& z\in\{0,1\}^n .
\end{aligned}
\end{equation}
By Proposition~\ref{prop:cor_principle}, formulation \eqref{eq:core_oc_formulation} preserves all globally optimal solutions of
\eqref{eq:core_template}. It may remove feasible but suboptimal points from the original formulation, but only those that violate the coordinate-wise optimality condition.

\subsection{The Closed-Form (CF) Component}

In many structured models, the optimality condition $0\in \partial_{x_i} f_i(x_i;x_{-i},y)$ yields an explicit minimizer for $x_i$ in terms of the remaining variables $(x_{-i}, y)$, that is, a closed-form function $x_i^\star: \mathbb{R}^{n-1}\times \mathbb{R}^m\to \mathbb{R}$ satisfying
\[
x_i^\star(x_{-i},y)
\in
\arg\min_{\xi\in\mathbb{R}} f_i(\xi;x_{-i},y).
\]
When such an expression is available, the on-regime value reduces to
\[
V_i^{\mathrm{on}}(x_{-i},y)
=
f_i(x_i^\star(x_{-i},y);x_{-i},y)+\lambda_i.
\]
The OC condition can then be combined with the closed-form coordinate minimizer, leading to the following off- and on-regime sets: 

\[
\begin{aligned}
\mathcal D_{i, \mathrm{CF}}^{0}
:=
\Bigl\{(x,y,z):\;& z_i=0,\; x_i=0,\\
& f_i(0;x_{-i},y)
\le
f_i(x_i^\star(x_{-i},y);x_{-i},y)+\lambda_i
\Bigr\},
\end{aligned}
\]
and
\[
\begin{aligned}
\mathcal D_{i, \mathrm{CF}}^{1}
:=
\Bigl\{(x,y,z):\;& z_i=1,\;
x_i=x_i^\star(x_{-i},y),\\
& f_i(x_i^\star(x_{-i},y);x_{-i},y)+\lambda_i
\le
f_i(0;x_{-i},y)
\Bigr\}.
\end{aligned}
\]
The resulting CF disjunction is
\[
(x,y,z)\in \mathcal D_{i, \mathrm{CF}}^{0}\cup \mathcal D_{i, \mathrm{CF}}^{1},
\qquad i=1,\ldots,n,
\]
and the corresponding CORe formulation is
\begin{equation}
\label{eq:core_oc_cf_formulation}
\begin{aligned}
\min_{x,y,z}\quad
& f(x,y)+\sum_{i=1}^n \lambda_i z_i\\
\text{s.t.}\quad
& (x,y,z)\in
\mathcal D_{i, \mathrm{CF}}^{0}\cup \mathcal D_{i, \mathrm{CF}}^{1},
\qquad i=1,\ldots,n,\\
& z\in\{0,1\}^n .
\end{aligned}
\end{equation}
Compared with the OC formulation, the CF component replaces the implicit
first-order condition by an explicit equality \(x_i=x_i^\star(x_{-i},y)\).

\subsection{The Disjunctive Reformulation Component}

In this section, we review how the disjunctive constraint in \eqref{eq:core_oc_formulation} or \eqref{eq:core_oc_cf_formulation} can be formulated in practice under certain conditions. To simplify notation, we write \(\mathcal D_i^0\) and \(\mathcal D_i^1\) to denote the two regime sets for coordinate \(i\), whether they arise from the OC formulation \((\mathcal D_{i,\mathrm{OC}}^0,\mathcal D_{i,\mathrm{OC}}^1)\) or from the CF formulation \((\mathcal D_{i,\mathrm{CF}}^0,\mathcal D_{i,\mathrm{CF}}^1)\). The disjunctive construction applies to either case.

Suppose $\mathcal D_i^{j}$ for $j=0,1$ admits a polyhedral representation $\mathcal D_i^{j}=\{(x,y,z): A_i^j (x,y)\le \mathbf b_i^j\}$, for constraint matrix $A_i^j$ and right-hand side vector $\mathbf b_i$ of appropriate dimension for $i=1,\dots,n,j=0,1$. Then, from disjunctive programming and the union of polyhedra \cite{balas1998disjunctive}, the disjunctive constraint can be equivalently represented by introducing two copies of vector $(x,y,z)$ as $(x^j,y^j,z^j), j=0,1$ and linear constraints \cite{balas1998disjunctive}

\[
\begin{aligned}
x=x^0+x^1,\\
y=y^0+y^1,\\
z^0+z^1=1,\\
z=z^1,\\
A_i^0x^0\le b_i^0z^0,\\
A_i^1x^1\le b_i^1z^1.\\
\end{aligned}
\]

In many applications, the on-regime with $z=1$ can be further partitioned into several subregimes. In this case, the disjunctive formulation can be extended similarly. Furthermore, if the sets defining the disjunctions are convex rather than polyhedral, then the union of convex sets may be represented similarly under certain conditions \cite{ceria1999convex}.

Overall, CORe provides a modeling framework for incorporating coordinate-wise optimality into indicator formulations. It starts with the coordinate subproblem, identifies regimes consistent with global optimality, and encodes these regimes via optimality-condition, closed-form, or disjunctive reformulations. The construction preserves all globally optimal solutions of the original model, while possibly removing feasible points that violate these necessary conditions. In the next section, we specify CORe reformulations of several classes of problems that have attracted attention in recent literature.

\begin{remark}[Role of variables $y$]
The CORe framework does not impose structural restrictions on the
variables $y$ beyond those required for convexity.
Additional convex constraints on $y$ can be incorporated without affecting
the validity of the coordinate-wise optimality conditions,
since these conditions are derived by fixing $(x_{-i},y)$
and comparing the active and inactive states of $x_i$.
Thus, the framework applies equally to models in which $y$ is constrained by general convex sets. In settings where the minimization over $y$ admits a closed-form minimizer for fixed $x$, one may further eliminate or explicitly characterize $y$ and substitute this expression into the formulation.
Such elimination is complementary to CORe and can yield additional enhancements, but it is not required for the validity of the framework.
\end{remark}

We next turn from the general CORe framework to concrete problem classes, demonstrating its effectiveness across a range of settings.

\section{CORe for Quadratic Programs with Indicators}
\label{sec:applications_QP}
In this section, we focus on quadratic programs with indicator variables, where the coordinate-wise structure is especially transparent. After fixing the remaining variables, each indicator-controlled coordinate reduces to a one-dimensional convex quadratic subproblem. Therefore, the active and inactive regimes can be compared in closed form, yielding tractable OC, CF, and disjunctive components of CORe.

We first derive the CORe reformulation from a coordinate-wise threshold rule. We then analyze structured sparse cases, such as when the sparsity pattern of the Hessian matrix corresponds to path and star graphs, to illustrate how the sparsity of the induced graph affects the worst-case behavior of branch-and-bound. Finally, we turn to quadratic programs with unbalanced bipartite structure, in which many indicator-controlled variables are coupled only through a few continuous variables, and show that CORe yields a polynomial bound on the size of the branch-and-bound tree.

\subsection{General Quadratic Programs with Indicators}
\label{subsec:quadratic}

Recall the mixed-integer convex quadratic program \eqref{eq:original_core}:
\begin{equation*}
\begin{aligned}
\min_{x\in\mathbb{R}^n,\; z\in\{0,1\}^n}\quad
& \tfrac12 x^\top Q x + c^\top x + \lambda^\top z \\
\text{s.t.}\quad
& -M_iz_i \leq x_i\leq M_i z_i, i=1,\ldots,n,
\end{aligned}
\end{equation*}
where $Q\succeq 0$, $Q_{ii}>0$ for all $i$, $\lambda_i\ge 0$, and $M_i$ is the big-M constant.

To apply CORe, fix a coordinate $i$ and treat all other variables $x_{-i}$ as given.
Define 
\begin{equation}
\label{eq:gdef_cor}
g_i := c_i + \sum_{j\neq i} Q_{ij} x_j .
\end{equation}
Eliminating the terms not involving $x_i$, the problem reduces to the one-dimensional problem
\begin{equation*}
\min_{x_i\in\mathbb{R},\, z_i\in\{0,1\}}
\;\frac12 Q_{ii}x_i^2 + g_i x_i + \lambda_i z_i
\quad
\text{s.t. } x_i(1-z_i)=0.
\end{equation*}

\paragraph{Optimality condition reformulation.}

The off-regime \(z_i=0\) forces \(x_i=0\), and therefore has value $V_i^{\mathrm{off}}(x_{-i})=0.$
In the on-regime \(z_i=1\), the coordinate \(x_i\) is free. Since \(Q_{ii}>0\), the one-dimensional quadratic function is minimized at $x_i^\star=-\frac{g_i}{Q_{ii}},$
with value $V_i^{\mathrm{on}}(x_{-i})
=
-\frac{g_i^2}{2Q_{ii}}+\lambda_i.$
Hence, the coordinate regime comparison gives
\[
V_i^{\mathrm{off}}(x_{-i})\le V_i^{\mathrm{on}}(x_{-i})
\quad\Longleftrightarrow\quad
|g_i|\le \tau_i,
\qquad
\tau_i:=\sqrt{2\lambda_i Q_{ii}}.
\]
Therefore, every globally optimal solution must satisfy the OC regime-selection logic
\[
z_i=0 \;\Longrightarrow\; |g_i|\le \tau_i,
\qquad
z_i=1 \;\Longrightarrow\; |g_i|\ge \tau_i.
\]
Here, the on-regime can be further divided into two subregimes: $g_i\ge \tau_i$ and $g_i\le -\tau_i$. Consequently, the OC  component induces the threshold regimes
\begin{align*}
    \mathcal D_{i, \mathrm{OC}}^0
&=
\{(x,y,z):\; z_i=0,\; x_i=0,\; |g_i|\le \tau_i\},\\
\mathcal D_{i, \mathrm{OC}}^{1}&=\mathcal D_{i, \mathrm{OC}}^{+}\cup \mathcal D_{i, \mathrm{OC}}^{-}, \mbox{where}\\
\mathcal D_{i, \mathrm{OC}}^{+}
&:=
\{(x,y,z):\; z_i=1,\; g_i\ge \tau_i\},\\
\mathcal D_{i, \mathrm{OC}}^{-}
&:=
\{(x,y,z):\; z_i=1,\; g_i\le -\tau_i\}.
\end{align*}

The OC component rules out activation patterns that are inconsistent with the
coordinate-wise regime comparison, but it does not yet enforce the coordinate
minimizer in the on-regime.

\paragraph{Closed-form reformulation.}

Because the active coordinate minimizer is available in closed form, we can further
refine the on-regimes. In the on-regime, coordinate-wise optimality requires $x_i=-\frac{g_i}{Q_{ii}},$
or equivalently, $Q_{ii}x_i+g_i=0.$
Adding this closed-form equality to the two active OC regimes yields the CF
disjunction
\begin{align*}
\mathcal D_{i, \mathrm{CF}}^0
&:= \{(x,y,z):\; z_i=0,\; x_i=0,\; |g_i|\le \tau_i\}, \\
\mathcal D_{i, \mathrm{CF}}^+
&:= \{(x,y,z):\; z_i=1,\; Q_{ii}x_i+g_i=0,\; g_i\ge \tau_i\}, \\
\mathcal D_{i, \mathrm{CF}}^-
&:= \{(x,y,z):\; z_i=1,\; Q_{ii}x_i+g_i=0,\; g_i\le -\tau_i\}.
\end{align*}
Hence every global minimizer of \eqref{eq:original_core} lies in
\[
\bigcap_{i=1}^n
\left(
\mathcal D_{i, \mathrm{CF}}^0
\;\cup\;
\mathcal D_{i, \mathrm{CF}}^+
\;\cup\;
\mathcal D_{i, \mathrm{CF}}^-
\right).
\]

\paragraph{Disjunctive reformulation.}

We next encode the three-term disjunction
\(\mathcal{D}_{i, \mathrm{CF}}^0 \cup \mathcal{D}_{i, \mathrm{CF}}^+ \cup \mathcal{D}_{i, \mathrm{CF}}^-\).
Introduce binary variables \(z_i^+,z_i^- \in \{0,1\}\), where
\(z_i^+\) and \(z_i^-\) indicate the positive- and negative-threshold regimes, $\mathcal{D}_{i, \mathrm{CF}}^+$ and $\mathcal{D}_{i, \mathrm{CF}}^-$,  respectively, and set $z_i = z_i^+ + z_i^- .$
The off-regime is therefore represented by \(z_i=0\). We also introduce disaggregated variables $(0,g_i^0), (x_i^+,g_i^+), (x_i^-,g_i^-),$
with aggregate relations
\[
x_i = x_i^+ + x_i^-,
\qquad
g_i = g_i^0 + g_i^+ + g_i^- .
\]
After intersecting each branch with the original bounds \(-M_iz_i \le x_i \le M_iz_i\), the resulting bounded disjunction can be encoded by the following constraints:
\begin{align}
&-M_iz^+_i \leq x^+_i \leq 0,\qquad  0 \leq x^-_i \leq M_iz^-_i \label{eq:bounds_x_sign_core_Q}\\
&  -\tau_i (1-z_i) \le g_i^0 \le \tau_i (1-z_i), \label{eq:h0_cor_Q}\\
& Q_{ii}x_i^+ + g_i^+ = 0, \qquad \tau_i z_i^+ \le g_i^+, \label{eq:h+_cor_Q}\\
& Q_{ii}x_i^- + g_i^- = 0, \qquad g_i^- \le -\tau_i z_i^-. \label{eq:h-_cor_Q}
\end{align}
Combining the above constraints with the coupling definition
\eqref{eq:gdef_cor} yields the CORe-enhanced formulation:
\begin{equation}
\label{eq:reformulated_core_general_QP}
\begin{aligned}
\min\quad
& \tfrac12 x^\top Q x + c^\top x
+ \lambda^\top z \\
\text{s.t.}\quad
& -M_iz_i \leq x_i\leq M_i z_i ,\qquad i=1,\ldots,n,\\
& g_i = c_i + \sum_{j\neq i} Q_{ij}x_j, \qquad i=1,\ldots,n,\\
& x_i = x_i^+ + x_i^-, \qquad i=1,\ldots,n,\\
& g_i = g_i^0 + g_i^+ + g_i^-, \qquad i=1,\ldots,n,\\
& z_i = z_i^+ + z_i^-, \qquad i=1,\ldots,n,\\
& \eqref{eq:bounds_x_sign_core_Q}-\eqref{eq:h-_cor_Q}, \qquad i=1,\ldots,n,\\
& z_i^+,z_i^- \in \{0,1\}, \qquad i=1,\ldots,n,.
\end{aligned}
\end{equation}

Formulation~\eqref{eq:reformulated_core_general_QP} is an instantiation of the general CORe framework of Section~\ref{sec:cor_framework}, where the CF component together with the disjunctive reformulation leads to tractable {\it linear} constraints. For each coordinate, constraints
\eqref{eq:bounds_x_sign_core_Q}--\eqref{eq:h-_cor_Q} give the standard extended
convex-hull formulation of the bounded three-term disjunction
\(\mathcal D_{i, \mathrm{CF}}^0 \cup \mathcal D_{i, \mathrm{CF}}^+ \cup \mathcal D_{i, \mathrm{CF}}^-\), where boundedness is
provided by the imposed bounds on \(x_i\) and the branch equalities
\(Q_{ii}x_i^\pm+g_i^\pm=0\). Thus, the formulation incorporates the local convex-hull descriptions of the CORe-induced disjunctions into the original indicator model.

\subsubsection{Role of Graph Sparsity}
The coordinate-wise conditions used by CORe are local with respect to the sparsity
graph of \(Q\), defined as an undirected graph \(G_Q=(V,E)\), where $V:=\{1,\dots,n\}$ and \((i,j)\in E\) if and only if \(Q_{ij}\neq 0\) for \(i\neq j\). For the quadratic model, the coordinate subproblem for variable
\(i\) depends on the other variables only through
\[
g_i(x):=c_i+\sum_{j\in N(i)}Q_{ij}x_j,
\]
where \(N(i)\) is the set of neighbors of \(i\) in $G_Q$. Thus the
activation regime of coordinate \(i\) is determined only by variables adjacent to
\(i\) in $G_Q$.

This locality can reduce the number of possible regime patterns when many
coordinates are controlled by a small set of variables. Suppose \(\mathcal I\) and \(\mathcal J\) are
sets of coordinates such that \(N(i)\subseteq \mathcal J\) for every \(i\in \mathcal I\). Then the
regimes of all coordinates in \(\mathcal I\) are determined by the hyperplanes
\[
c_i+\sum_{j\in \mathcal J}Q_{ij}x_j=\pm \tau_i,\qquad i\in \mathcal I,
\]
in the lower-dimensional space of the variables \(x_\mathcal J:=(x_j)_{j\in \mathcal J}\), where \(\tau_i\) is the
activation threshold for coordinate \(i\). Hence, a classical hyperplane-arrangement bound  says that $2|\mathcal I|$ affine hyperplanes in $\mathbb{R}^{|\mathcal J|}$
 divide the space into at most
\[
\sum_{\ell=0}^{|\mathcal J|}\binom{2|\mathcal I|}{\ell}
\]
distinct regime patterns \cite{zaslavsky1975facing}. 
In particular, if \(|\mathcal J|\) is fixed, this number is polynomial in \(|\mathcal I|\), and if
\(|\mathcal J|=1\), it is linear. 

This intuitively explains why sparse structures, such as star or bipartite sparsity graphs, can be favorable for CORe.
In a star graph, all leaf regimes are controlled by the center node $r$, yielding $\mathcal J=\{r\}$ and $\mathcal I = [n]\backslash\{r\}$. Therefore, the
leaf thresholds form a one-dimensional arrangement on this centered variable with only \(O(n)\) intervals. Note that a star graph may be viewed as a bipartite graph with one node (center) in one partition and the remaining nodes in the other partition. More generally, in an unbalanced bipartite graph, the regimes on the larger partition are controlled by a low-dimensional hyperplane arrangement defined over the smaller partition. Sparsity alone, however, is not sufficient: in a path graph, coordinate conditions are governed by overlapping local neighborhoods rather than by a single small set of variables. Thus, the computational benefit of sparsity stems from low-dimensional control over many local regimes, not merely from having few nonzero entries in $Q$. These structures are studied in detail in Sections~\ref{subsec:structured_sparse_qp} and~\ref{subsec:bipartite_qp}, which explain the substantial computational improvements observed in sparse instances while admitting difficult sparse cases in the worst case.

\subsubsection{Role of Regularization}

There are two notions of sparsity at play for mixed-integer convex quadratic programs with indicators: one is that of the sparsity of the Hessian matrix $Q$, and the other is that of the solution sparsity modeled with the regularizer involving the indicator variables.  
In the quadratic case, the CORe activation threshold is \(\tau_i=\sqrt{2\lambda_i Q_{ii}}\). Thus, increasing
\(\lambda_i\) enlarges the off-regime region \(|g_i|\le \tau_i\), making it easier
to certify coordinates as inactive. This does not change the objective relative to the
original formulation or alter the set of global optima; rather, it makes the coordinate-wise optimality conditions more informative precisely when the optimal solution is sparser. Consequently, CORe tends to benefit more from stronger regularization than
the original big-\(M\) formulation, which only captures the larger penalty through the
objective and does not explicitly encode the threshold rule. This distinction will become evident in the experiments presented in Section~\ref{sec:experiments}.

\subsection{Structured Sparse Quadratic Programs with Indicators}
\label{subsec:structured_sparse_qp}

The CORe formulation derived in Section~\ref{subsec:quadratic} is independent of the particular sparsity graph of \(Q\). However, the effect of sparsity on the strength of the CORe formulation has not yet been quantified. In this section, we address this question by analyzing two special classes of sparsity patterns: star- and path-structured problems. Both belong to the broader class of tree-structured problems and both are known to be solvable in polynomial time~\cite{liu2023graph}. Despite this shared tractability, we show that their behavior within a branch-and-bound framework can differ dramatically when the CORe formulation is used. In particular, for star-structured problems, the size of the branch-and-bound tree grows only linearly with the problem dimension, whereas for path-structured problems, it can grow exponentially in the worst case.

\subsubsection{Star Graph}

We show that when the sparsity graph of $Q$ is a star, the CORe formulation admits an $O(n)$ bound on the size of the branch-and-bound tree. This is a substantial improvement over the big-M formulation, whose tree can have
$O(2^n)$ nodes.

\begin{proposition}
\label{prop:star_graph}
Suppose the sparsity graph of $Q$ is a star with a center node $r$. Then the number of feasible CORe regime patterns over the leaf coordinates is $O(n)$. Consequently, a branching algorithm that branches on unresolved disjunctions $\mathcal D_{i, \mathrm{CF}}^0 \cup \mathcal D_{i, \mathrm{CF}}^+ \cup \mathcal D_{i, \mathrm{CF}}^-$ for each $i \ne r$ in the fixed order $i=1,\dots,n$ terminates after $O(n)$ nodes. 

\end{proposition}

\begin{proof}
\label{sec:proof_star}

Suppose the sparsity graph of \(Q\) is a star on the vertex set \([n]\) with center node \(r\), and let \( [n]\setminus\{r\}\) denote the set of graph leaves.
For each leaf $i \neq r$, the only neighbor of \(i\) is the center node \(r\). Hence we have
\[
g_i = c_i + Q_{ir} x_r,
\]
so the disjunction $\mathcal D_{i, \mathrm{CF}}^0 \cup \mathcal D_{i, \mathrm{CF}}^+ \cup \mathcal D_{i, \mathrm{CF}}^-$ depends only on the scalar variable $x_r$. The three regimes correspond to the intervals
$$
\begin{aligned}
S_i^- &:= \{x_r : c_i + Q_{ir} x_r \le -\tau_i\}, \\
S_i^0 &:= \{x_r : |c_i + Q_{ir} x_r| \le \tau_i\}, \\
S_i^+ &:= \{x_r : c_i + Q_{ir} x_r \ge \tau_i\}.
\end{aligned}
$$

Suppose, without loss of generality, that we branch on the disjunctions  $\mathcal D_{i, \mathrm{CF}}^0 \cup \mathcal D_{i, \mathrm{CF}}^+ \cup \mathcal D_{i, \mathrm{CF}}^-$
 in the order $i=1,\ldots,n$ for $i\ne r$. Each disjunction induces at most two breakpoints in $x_r$, given by $c_i + Q_{ir} x_r = \pm \tau_i$. Hence all disjunctions taken together at the leaf nodes of the branch-and-bound tree partition $\mathbb{R}$ into at most $2(n-1)+1 = O(n)$ intervals. On each such interval, the regime of every leaf is fixed, and therefore all variables $x_i$, $i \neq r$, are uniquely determined (either zero or affine in $x_r$). Thus, each interval corresponds to a feasible terminal node, and the number of feasible branch-and-bound leaves is at most $L = O(n)$.

Next, consider the structure of the branch-and-bound tree. Each node representing a disjunction for an index $i$ corresponds to a feasible interval $I:=[L,U] \subseteq \mathbb{R}$ for $x_r$ determined by the intersection of the resolved disjunctions at this stage of the branch-and-bound search that lead to a feasible subproblem, i.e., $I=\cap_{j=1}^{i-1}S_j^{k}$ where $k\in\{+,-,0\}$ depending on the branching sequence until that node. Branching on a coordinate $i$ replaces $I$ by the three sets
\begin{equation*}
    \begin{aligned}
    I^- &:= I \cap S_i^-= \{x_r\in[L,U] : c_i + Q_{ir} x_r \le -\tau_i\}, \\
I^0 &:= I \cap S_i^0= \{x_r\in[L,U] : |c_i + Q_{ir} x_r| \le \tau_i\}, \\
I^+ &:= I \cap S_i^+= \{x_r\in[L,U]  : c_i + Q_{ir} x_r \ge \tau_i\}.
\end{aligned}
\end{equation*}

Since $S_i^-, S_i^0, S_i^+$ are three consecutive intervals covering $\mathbb{R}$, their intersections with $I$ form at most three consecutive subintervals of $I$ given by $I_i^-, I_i^0, I_i^+$, at least one of which is a feasible interval.

Let $N_{\mathrm{leaf}}^{\mathrm{feas}}$ denote the number of feasible leaves of the final branch-and-bound tree. Then $N_{\mathrm{leaf}}^{\mathrm{feas}} \le O(n).$
Since every intermediate node has at least one feasible child, the number of non-leaf nodes is at most $N_{\mathrm{leaf}}^{\mathrm{feas}}  = O(n)$. Moreover, since each non-leaf node generates at most two children with infeasible intervals, the number of infeasible leaves is also at most $O(n)$.

Combining these bounds, the total number of nodes in the branch-and-bound tree satisfies $N_{\mathrm{total}} = O(n).$ \qed
\end{proof}

The linear bound in Proposition~\ref{prop:star_graph} applies to the tailored regime-branching scheme analyzed above, rather than to the standard branch-and-bound procedure used by commercial solvers. Nevertheless, in the computational experiments of Section~\ref{sec:experiments}, we solve the CORe formulation using Gurobi with its default binary-variable branching strategy. The tested star-graph instances are essentially solved at the root node, with negligible root gaps and only a small number of nodes explored, even for problems of size \(n=10^4\). Moreover, in these instances, CORe is faster than the specialized dynamic-programming benchmark, suggesting that the coordinate-wise reformulation allows a general-purpose solver to exploit the same low-dimensional structure with very little branching overhead.

\begin{remark}
The three-way disjunction in the proposition statement can also be implemented using standard binary-variable branching with a specialized branching rule. Once a coordinate $i$ is selected, one may first branch on $z_i^+$. The branch $z_i^+=1$ fixes $z_i^-=0$ and $z_i=1$, and therefore corresponds to $\mathcal D_i^+$. On the other branch $z_i^+=0$, one then branches on $z_i^-$. The child node $z_i^-=0$ corresponds to $\mathcal D_i^0$, while the child node $z_i^-=1$ corresponds to $\mathcal D_i^-$. Thus, this two-step binary branching rule partitions the feasible region in the same way as the three-way disjunction, up to a constant factor in the number of branch-and-bound nodes. The key requirement for the counting argument is that, once coordinate $i$ is selected, it is fully resolved into one of the three regimes $\mathcal D_i^+$, $\mathcal D_i^0$, or $\mathcal D_i^-$.
\end{remark}

\subsubsection{Path Graph} 
Our next proposition shows that, when the sparsity graph of $Q$ is a connected path, the size of the branch-and-bound tree may grow exponentially in $n$. This in turn implies that Proposition~\ref{prop:star_graph} cannot be extended to general trees.
\begin{proposition}
\label{prop:path_counterexample_exp}
There exists a family of instances of \eqref{eq:original_core} on path graphs for which the number of feasible integer leaves of the branch-and-bound tree applied to the CORe formulation is exponential in \(n\).

\end{proposition}

The proof of Proposition~\ref{prop:path_counterexample_exp} is provided in Appendix~\ref{sec:proof_counter}. Despite this negative worst-case result, we note that it does not reflect the typical behavior observed in our experiments. On the random path-graph instances in Section~\ref{sec:experiments}, CORe solves all tested instances essentially at the root node, with zero root gap and roughly one explored node on average up to \(n=10^4\), and is faster than the specialized dynamic-programming benchmark on the reported instances. Thus, although Proposition~\ref{prop:path_counterexample_exp} rules out a uniform polynomial bound for branch-and-bound on all path graphs, the CORe formulation remains highly effective on the structured random instances considered in our computational study.

\subsection{Quadratic Programs with Unbalanced Bipartite Structure}
\label{subsec:bipartite_qp}

In this section, we study mixed-integer convex quadratic programs in which a large number of indicator-controlled variables interact only through a small number of additional continuous variables. Consider
\begin{subequations}
\label{problem:low_rank_general}
    \begin{align}
        \min_{x\in \mathbb{R}^n ,y\in \mathbb{R}^m, z\in\{0, 1\}^n} \quad
        & \frac{1}{2} x^{\top} D x+\frac{1}{2} y^{\top} G y
        + x^{\top} F y + a^\top x + b^\top y +\lambda^\top z \label{eq:obj}\\
        \text{s.t.}\quad
        & -M_iz_i \leq x_i\leq M_i z_i, \qquad i=1,\ldots,n,
    \end{align}
\end{subequations}
where \(D\in \mathbb{R}^{n\times n}\) is diagonal with positive diagonal entries, \(G\in \mathbb{R}^{m\times m}\), \(F\in \mathbb{R}^{n \times m}\) with \(m\ll n\), and
\[
    \begin{pmatrix}
        D & F\\
        F^\top & G
    \end{pmatrix}
    \succeq 0.
\]
We refer to \eqref{problem:low_rank_general} as an \emph{unbalanced bipartite} formulation: the sparsity graph of the quadratic form in \((x,y)\) is bipartite with partitions of size \(n\) and \(m\), so the \(n\) indicator-controlled coordinates \(x_i\) interact only through the \(m\) continuous variables \(y\). Note that \eqref{problem:low_rank_general} generalizes the star case, as it reduces precisely to this case when $m=1$. In Appendix~\ref{sec:poly-algo}, we show that the bipartite structure can be exploited by a geometric pruning algorithm with \(O(n^{m+1})\) time and memory complexity for fixed \(m\).

We apply CORe to problem~\eqref{problem:low_rank_general}. Fix an index $i$ and treat
$(x_{-i},y)$ as given. Since $D$ is diagonal, the objective terms involving $x_i$
reduce to a one-dimensional convex quadratic. Define the scalar
\[
g_i(y) := F_i y + a_i,
\qquad
d_i := D_{ii}>0,
\]
where $F_i$ denotes the $i$th row of $F$.
Up to terms independent of $(x_i,z_i)$, the coordinate-restricted subproblem is
\begin{equation*}
\min_{x_i\in\mathbb{R},\, z_i\in\{0,1\}}
\ \frac12 d_i x_i^2 + g_i(y)\,x_i + \lambda_i z_i
\quad \text{s.t.}\quad x_i(1-z_i)=0.
\end{equation*}

\paragraph{Optimality condition reformulation.}
The off-regime $z_i=0$ forces $x_i=0$, hence $V_i^{\mathrm{off}}(y)=0.$
In the on-regime $z_i=1$, the unconstrained minimizer is $x_i^\star(y)= -\frac{g_i(y)}{d_i},$
and the optimal on-regime value equals
\[
V_i^{\mathrm{on}}(y)
=\min_{x_i\in\mathbb{R}}\Big(\frac12 d_i x_i^2 + g_i(y)\,x_i\Big)+\lambda_i
= -\frac{g_i(y)^2}{2d_i}+\lambda_i.
\]
Therefore, the regime comparison $V_i^{\mathrm{off}}(y)\le V_i^{\mathrm{on}}(y)$ is equivalent to the
threshold condition
\begin{equation*}
|g_i(y)| \le \tau_i,
\qquad
\tau_i := \sqrt{2d_i\lambda_i}.
\end{equation*}
In other words, every globally optimal solution satisfies the implications
\begin{equation*}
z_i=0 \ \Rightarrow\ |g_i(y)|\le \tau_i,
\qquad
z_i=1 \ \Rightarrow\ |g_i(y)|\ge \tau_i,
\end{equation*}
which eliminates regime assignments inconsistent with coordinate-wise optimality.

\paragraph{Closed-form reformulation.}
Since $x_i^\star(y)$ is available in closed form, CORe additionally yields

\begin{equation}
\label{eq:cf_imp_lowrank}
z_i=1 \ \Rightarrow\ x_i = -\frac{g_i(y)}{d_i},
\qquad
z_i=0 \ \Rightarrow\ x_i=0.
\end{equation}
Imposing \eqref{eq:cf_imp_lowrank} tightens the on-regime by enforcing the
coordinate-wise minimizer whenever $z_i=1$.
For fixed $x$, the minimization over $y$ also admits the closed-form solution $y(x) = -G^{-1}(F^\top x + b)$; consistent with the remark on the role of the $y$ variables in Section~\ref{sec:cor_framework}, we do not substitute this expression into the formulation, since retaining $y$ is precisely what keeps the coordinate couplings low-dimensional.

\paragraph{Disjunctive reformulation.}

For each coordinate $i$, the OC threshold
$|g_i(y)| \le \tau_i$ and the CF mapping
$x_i = -g_i(y)/d_i$ induce the following three-term disjunction:
\[
|g_i(y)| \le \tau_i,\ x_i=0
\quad \vee \quad
g_i(y) \ge \tau_i,\ x_i=-\frac{g_i(y)}{d_i}
\quad \vee \quad
g_i(y) \le -\tau_i,\ x_i=-\frac{g_i(y)}{d_i}.
\]

To obtain a bounded disjunctive formulation, suppose that valid bounds
on the affine quantities $g_i(y)=F_i y+a_i$ are available:
\[
L_i \le g_i(y) \le U_i, \qquad i=1,\dots,n.
\]
If no finite bound is available for a given coordinate, the corresponding bound may be omitted, equivalently taking $L_i=-\infty$ or $U_i=\infty$.

Introduce binary variables $z_i^+, z_i^- \in \{0,1\}$ indicating the two on-regimes, and set $z_i = z_i^+ + z_i^-$ together with $z_i \le 1$, so that the off-regime corresponds to $z_i = 0$. We disaggregate $g_i$ as $t_i = g_i^0 + g_i^+ + g_i^-$, with one copy per regime; since $x_i = 0$ in the off-regime, it suffices to disaggregate $x_i = x_i^+ + x_i^-$. Intersecting each regime with the bounds $L_i \le g_i(y) \le U_i$ and $-M_i z_i \le x_i \le M_i z_i$ yields the regime-local constraints
\begin{align*}
-M_iz^+_i \leq x^+_i \leq 0,\qquad & 0 \leq x^-_i \leq M_iz^-_i, \\
-\tau_i (1-z_i) \le g_i^0 \le \tau_i (1-z_i), \qquad & \\
\tau_i z_i^+ \le g_i^+ \le U_i z_i^+, \qquad & x_i^+ = -\frac{1}{d_i}g_i^+,\\
L_i z_i^- \le g_i^- \le -\tau_i z_i^-, \qquad & x_i^- = -\frac{1}{d_i}g_i^-.
\end{align*}

Applying this construction for all $i$ yields a CORe-enhanced formulation
that tightly encodes the coordinate-wise optimality structure
of problem~\eqref{problem:low_rank_general}:

\begin{subequations}\label{prob:cor_low_rank}
	\begin{align}
    \min_{\substack{y\in\mathbb{R}^m\\
    z\in [0,1]^n,\, z^+,z^-\in\{0,1\}^n\\
    g^0, g^+, g^- \in \mathbb{R}^n\\
    x^+, x^-, x \in \mathbb{R}^n}}  & \frac{1}{2} x^\top D x + \frac{1}{2} y^\top G y + x^\top F y + a^\top x + b^\top y + \lambda^\top z\\
		\text{s.t. }
         &-M_iz_i \leq x_i \leq M_iz_i, \ i=1,\dots,n,\\
        & z_i = z_i^+ + z_i^-, \ i=1,\dots,n,\\
        & F_i y + a_i = g_i^0 + g_i^+ + g_i^-, \ i=1,\dots,n,\\
        & x_i = x_i^+ + x_i^-, \ i=1,\dots,n,\\
        &-M_iz^+_i \leq x^+_i \leq 0,\;  0 \leq x^-_i \leq M_iz^-_i, \ i=1,\dots,n,\\
        & -\tau_i (1-z_i) \le g_i^0 \le \tau_i (1-z_i),  \ i=1,\dots,n,\\
& \tau_i z_i^+ \le g_i^+ \le U_i z_i^+, \ x_i^+ = -\frac{1}{D_{ii}}\, g_i^+, \ i=1,\dots,n,\\
& L_i z_i^- \le g_i^- \le -\tau_i z_i^-, \ x_i^- = -\frac{1}{D_{ii}}\, g_i^-, \ i=1,\dots,n.
	\end{align}
\end{subequations}

Note that the constraint $z_i \le 1$, implicit in the declaration $z\in[0,1]^n$, excludes $z_i^+ = z_i^- = 1$; hence $z_i \in \{0,1\}$ holds automatically and $z$ need not be declared binary. In formulation~\eqref{prob:cor_low_rank}, the constants \(M_i\) are the same as in formulation~\eqref{problem:low_rank_general}, and the bounds \(L_i,U_i\), when available, strengthen the disaggregated representation of the coordinate-wise optimality regimes.

The next result bounds the size of the branch-and-bound tree for the CORe formulation \eqref{prob:cor_low_rank} by counting the cells of the hyperplane arrangement that the coordinate thresholds induce in the space of the $y$ variables.

\begin{proposition}[Polynomial tree bound for fixed $m$]
\label{prop:poly_tree}
Suppose formulation \eqref{prob:cor_low_rank} is solved by a branch-and-bound scheme that branches on the regime variables \(z_i^+\) and \(z_i^-\) and prunes infeasible nodes. If \(m\) is fixed, then the branch-and-bound tree contains \(O(\min\{4^n,n^{m+1}\})\) nodes.
\end{proposition}

\begin{proof}
There are $2n$ binary variables, so the branch-and-bound tree has depth at most $2n$ and hence $O(4^n)$ nodes.

For the second bound, each coordinate $i$ induces two hyperplanes in the $y$-space,
\[
F_i y+a_i=\pm\tau_i,\qquad i=1,\dots,n.
\]
These $2n$ hyperplanes partition $\mathbb{R}^m$ into at most $\sum_{\ell=0}^{m}\binom{2n}{\ell}$ cells \cite{zaslavsky1975facing}, which is $O(n^m)$ for fixed $m$. Inside any cell, the signs of $F_i y+a_i-\tau_i$ and $F_i y+a_i+\tau_i$ are fixed, and therefore the regime of every coordinate is uniquely determined; distinct feasible integer assignments of the regime variables thus correspond to distinct cells, and their number is bounded by $O(n^m)$.

Next, we count the nodes of the branch-and-bound tree. The nodes whose relaxations are feasible form a subtree containing the root. A leaf of this subtree either has all regime variables resolved, or is an internal node of the branch-and-bound tree all of whose children are infeasible; the latter cannot occur, because a feasible node contains a point $\bar y$ in some cell of the arrangement, and the child obtained by fixing the next regime variable consistently with the regime pattern of that cell is again feasible. Hence every leaf of the feasible subtree corresponds to a fully resolved regime assignment, and the subtree has at most $O(n^m)$ leaves. Since its depth is at most $2n$, the feasible subtree contains at most $O(n\cdot n^{m})$ nodes. Finally, every node of the branch-and-bound tree is either in the feasible subtree or is a child of a node in it (an infeasible node is pruned immediately and has no children), and each node has at most two children. Therefore, the total number of nodes is $O(n^{m+1})$. \qed
\end{proof}

\section{CORe for Robust Single-Index Models}
\label{sec:single_index}
Moving beyond quadratic models, we next apply CORe to a general class of robust single-index models. Consider the optimization problem
\begin{equation}
\label{eq:sim_original_cor}
\min_{\theta \in \mathbb{R}^d,\; z \in \{0,1\}^n}
\;\sum_{i=1}^n \ell\!\left(r_i,\; f(\langle \phi_i,\theta\rangle)\right)(1-z_i) + \frac{1}{d}\|\theta\|_2^2
+ \sum_{i=1}^n \lambda_i z_i,
\end{equation}
where $\phi_i \in \mathbb{R}^d$ are covariates,
$r_i \in \mathbb{R}$ are responses,
$f:\mathbb{R}\to\mathbb{R}$ is a known link function,
$\ell(\cdot,\cdot)$ is a loss function,
and $z_i$ indicates whether observation $i$ is treated as an outlier. 
In this form, the indicators switch objective terms on and off rather than controlling continuous variables, so \eqref{eq:sim_original_cor} does not immediately fit template \eqref{eq:core_template}; we first rewrite it in indicator-constraint form. Two special cases of \eqref{eq:sim_original_cor} will play a central role in this section: penalized least trimmed squares, obtained with the identity link and quadratic loss (Section~\ref{subsec:penalized_lts}), and robust logistic regression, obtained with the sigmoid link and cross-entropy loss (Section~\ref{subsec:classification}). In Section~\ref{subsec:sim_poly}, we show that when the number of covariates $d$ is fixed, the coordinate-optimality structure yields a polynomial bound on the size of the branch-and-bound tree for the entire class, together with a polynomial-time solution scheme under a mild convexity condition.

\begin{assumption}
\label{assump:loss_cor}
For each $i$, the function
$\ell_i:=\ell(r_i,\cdot)$
is convex, nonnegative, and satisfies $\ell_i(r_i)=0$. Moreover, the link function \(f\) is monotone.
\end{assumption}

Assumption~\ref{assump:loss_cor} is mild and is satisfied by most common choices of loss and link functions. On the loss side, it is satisfied by quadratic, Huber, and cross-entropy losses, covering standard regression, robust regression, and classification settings, respectively. On the link function side, it is satisfied by the identity, ReLU, and sigmoid, corresponding to linear, rectified, and logistic models, respectively.

To expose coordinate-wise structure, we first provide a reformulation of~\eqref{eq:sim_original_cor}.
Under Assumption~\ref{assump:loss_cor}, and for constants $M_i \ge \sup_\theta\, |r_i - f(\langle \phi_i,\theta\rangle)|$, where the supremum is taken over the relevant range of $\theta$ (see the bounds derived in Section~\ref{subsec:classification}), we can rewrite \eqref{eq:sim_original_cor} as
\begin{equation}
\label{eq:sim_w_cor}
\min_{\theta,w,z}
\;\sum_{i=1}^n \ell_i\!\left(f(\langle \phi_i,\theta\rangle)+w_i\right) + \frac{1}{d}\|\theta\|_2^2
+ \sum_{i=1}^n \lambda_i z_i
\quad
\text{s.t. } -M_iz_i \leq w_i \leq M_iz_i.
\end{equation}
To see this equivalence, note that $z_i = 0$ enforces $w_i = 0$ in~\eqref{eq:sim_w_cor}, so the corresponding objective term reduces to $\ell_i(f(\langle \phi_i, \theta\rangle))$, matching~\eqref{eq:sim_original_cor}. When $z_i = 1$ and $M_i$ is sufficiently large, the variable $w_i$ takes the value $w_i = r_i - f(\langle \phi_i, \theta\rangle)$ at optimality, which gives $\ell_i(f(\langle \phi_i, \theta\rangle) + w_i) = 0$ by Assumption~\ref{assump:loss_cor}, effectively removing $r_i$ from the objective.

Under Assumption~\ref{assump:loss_cor}, the set
\[
S_i(\lambda_i):=\{t\in\mathbb R:\ell_i(f(t))\le \lambda_i\}
\]
is a (possibly unbounded) interval. This follows from the convexity of $\ell_i$, which ensures that all sublevel sets of $\ell_i$ are intervals, and the fact that the preimage of an interval under a monotone function is again an interval.

As we show next, this interval structure is exactly what makes the CORe regime comparison representable by linear constraints.

\paragraph{Optimality condition reformulation.}

Fix \(\theta\) and consider the coordinate subproblem in \((w_i,z_i)\):
\[
\min_{w_i,\,z_i\in\{0,1\}}
\;\ell_i\!\left(f(\langle \phi_i,\theta\rangle)+w_i\right)
+\lambda_i z_i
\quad
\text{s.t.}\quad
-M_iz_i\le w_i\le M_iz_i .
\]
If \(z_i=0\), then \(w_i=0\), and the off-regime value is $V_i^{\mathrm{off}}(\theta)
=
\ell_i\!\left(f(\langle \phi_i,\theta\rangle)\right).$
If \(z_i=1\), then \(w_i\) is free within the big-\(M\) bounds, and a minimizer is $w_i^\star(\theta)
=
r_i-f(\langle \phi_i,\theta\rangle)$,
which is feasible by the choice of $M_i$. Hence, since $\ell_i\ge 0$ and $\ell_i(r_i)=0$ by Assumption~\ref{assump:loss_cor}, the on-regime has value $V_i^{\mathrm{on}}(\theta)=\lambda_i$.
Thus, the coordinate-wise comparison is
\[
V_i^{\mathrm{off}}(\theta)\le V_i^{\mathrm{on}}(\theta)
\quad\Longleftrightarrow\quad
\ell_i\!\left(f(\langle \phi_i,\theta\rangle)\right)\le \lambda_i\quad\Longleftrightarrow\quad \langle \phi_i,\theta\rangle\in S_i(\lambda_i).
\]
Recall that $S_i(\lambda_i)$ is an interval; write its endpoints as $a_i \in \mathbb{R} \cup \{-\infty\}$ and $b_i \in \mathbb{R} \cup \{+\infty\}$, with the interval open at any infinite endpoint. The coordinate-wise comparison induces three regimes:
\begin{align*}
    \mathcal D_{i, \mathrm{OC}}^0 &= \{(\theta, w,z): z_i=0,\; w_i=0,\; a_i\le \langle \phi_i,\theta\rangle\le b_i \},\\
    \mathcal D_{i, \mathrm{OC}}^- &= \{(\theta, w,z): z_i=1,\; \langle \phi_i,\theta\rangle\le a_i\}, \\
    \mathcal D_{i, \mathrm{OC}}^+ &= \{(\theta, w,z): z_i=1,\; \langle \phi_i,\theta\rangle\ge b_i\}.
\end{align*}
If $a_i = -\infty$ or $b_i = +\infty$, the corresponding sets $\mathcal{D}_{i, \mathrm{OC}}^-$ or $\mathcal{D}_{i, \mathrm{OC}}^+$ are empty.

We next encode the disjunction $\mathcal D_{i, \mathrm{OC}}^0\cup \mathcal D_{i, \mathrm{OC}}^-\cup \mathcal D_{i, \mathrm{OC}}^+$.
Let $t_i=\langle \phi_i,\theta\rangle$,
and assume known bounds \(L_i\le t_i\le U_i\). Let the binary variables
\(z_i^-,z_i^+\in\{0,1\}\) represent the left and
right on-regimes, respectively, and set $z_i=z_i^-+z_i^+$.
The off-regime is represented by \(z_i=0\). Finally, we introduce the disaggregated variables $t_i^0, t_i^-, t_i^+,$
with $t_i=t_i^0+t_i^-+t_i^+.$
After intersecting the three regimes with the bounds \(L_i\le t_i\le U_i\), the bounded disjunction can be encoded as
\begin{align}
& a_i(1-z_i)\le t_i^0\le b_i(1-z_i), \label{eq:sim_t0_core}\\
& L_i z_i^-\le t_i^-\le a_i z_i^-, \label{eq:sim_tminus_core}\\
& b_i z_i^+\le t_i^+\le U_i z_i^+. \label{eq:sim_tplus_core}
\end{align}
Combining these constraints with the original formulation yields the CORe-enhanced
formulation
\begin{equation}
\label{eq:sim_core_reformulation}
\begin{aligned}
\min_{\theta,w,z,z^-,z^+,t,t^0,t^-,t^+}\quad
& \sum_{i=1}^n
\ell_i\!\left(f(t_i)+w_i\right)
+\frac{1}{d}\|\theta\|_2^2
+\sum_{i=1}^n \lambda_i z_i \\
\text{s.t.}\quad
& t_i=\langle \phi_i,\theta\rangle,
\qquad i=1,\ldots,n,\\
& -M_iz_i\le w_i\le M_iz_i,
\qquad i=1,\ldots,n,\\
& z_i=z_i^-+z_i^+,
\qquad i=1,\ldots,n,\\
& t_i=t_i^0+t_i^-+t_i^+,
\qquad i=1,\ldots,n,\\
& \eqref{eq:sim_t0_core}-\eqref{eq:sim_tplus_core},
\qquad i=1,\ldots,n,\\
& z_i^-,z_i^+\in\{0,1\},
\qquad i=1,\ldots,n.
\end{aligned}
\end{equation}

In the above CORe-enhanced formulation, all constraints are linear except for the integrality constraints on $(z^-,z^+)$; the formulation is therefore directly amenable to branch-and-bound whenever the composite terms $\ell_i(f(t_i)+w_i)$ are convex, as is the case for affine links. For nonlinear links the composition is nonconvex in general; however, convexity can sometimes be recovered by eliminating the correction variables $w$ and expressing the loss directly in terms of $t_i$, as we illustrate for robust logistic regression in Section~\ref{subsec:classification}.

\begin{remark}\label{remark:CF-single-index}
    Although $w_i$ admits a closed-form minimizer $w_i^\star(\theta) = r_i - f(\langle \phi_i,\theta\rangle)$ on the on-regime, we do not encode this equality directly in the formulation: unless $f$ is affine, the constraint $w_i = r_i - f(\langle \phi_i,\theta\rangle)$ is nonlinear and does not yield a convex set. This highlights a key limitation of CORe for general robust single-index models relative to quadratic programs with indicators. For the latter, CORe can be strengthened by both the optimality-condition and closed-form reformulations; for the former, only the optimality-condition reformulation is tractable. As a result, CORe is expected to be less effective on general robust single-index models.
\end{remark}

\subsection{Penalized Least Trimmed Squares}
\label{subsec:penalized_lts}

An important special case of \eqref{eq:sim_original_cor} arises with the identity link $f(t)=t$ and the quadratic loss $\ell(r,u)=\frac12(r-u)^2$, for which the model reads
\begin{equation}
\label{eq:penalized_lts}
\min_{\theta\in\mathbb{R}^d,\; z\in\{0,1\}^n}
\;\sum_{i=1}^n \tfrac12\bigl(r_i-\langle \phi_i,\theta\rangle\bigr)^2(1-z_i)
+\frac{1}{d}\|\theta\|_2^2
+\sum_{i=1}^n \lambda_i z_i .
\end{equation}
Problem \eqref{eq:penalized_lts} is a penalized form of least trimmed squares (LTS) regression: each observation is either fit by least squares or discarded at cost $\lambda_i$, so the estimator automatically trims observations whose residuals are too large to be worth fitting \cite{zioutas2005deleting,gomez2021outlier,meng2026computation}. Assumption~\ref{assump:loss_cor} holds, $\ell_i(f(t))=\frac12(r_i-t)^2$, and the interval in the regime comparison is
\[
S_i(\lambda_i)=\bigl[\,r_i-\sqrt{2\lambda_i},\; r_i+\sqrt{2\lambda_i}\,\bigr],
\]
i.e., observation $i$ is fit at an optimal solution if its index $\langle \phi_i,\theta\rangle$ lies within $\sqrt{2\lambda_i}$ of the response $r_i$, and trimmed otherwise. Moreover, in its lifted form \eqref{eq:sim_w_cor}, problem \eqref{eq:penalized_lts} becomes
\begin{equation}
\label{eq:lts_lifted}
\min_{\theta,\,w,\,z}\;
\tfrac12\|r-\Phi\theta-w\|_2^2+\frac{1}{d}\|\theta\|_2^2+\lambda^\top z
\quad\text{s.t.}\quad -M_iz_i\le w_i\le M_iz_i,\; i=1,\ldots,n,
\end{equation}
where $\Phi\in\mathbb{R}^{n\times d}$ collects the covariates $\phi_1,\ldots,\phi_n$ as rows. Formulation \eqref{eq:lts_lifted} is precisely an instance of the unbalanced bipartite quadratic model \eqref{problem:low_rank_general} of Section~\ref{subsec:bipartite_qp}, with the correction variables $w$ in the role of $x$ and the regression coefficients $\theta$ in the role of $y$. The developments of this section can therefore be seen as an extension of Section~\ref{subsec:bipartite_qp} beyond quadratic objectives.

\paragraph{Relation to \cite{meng2026computation}.}
The closest work to this section is \cite{meng2026computation}, which develops a branch-and-bound framework for least trimmed squares enhanced by hyperplane arrangements. The two approaches share a central insight: at an optimal solution, whether observation $i$ is trimmed is determined by the position of the index $\langle \phi_i,\theta\rangle$ relative to explicit thresholds, and the arrangement of the resulting $2n$ threshold hyperplanes in the $d$-dimensional coefficient space bounds the number of trimming patterns that can occur at optimality by $O(n^d)$ for fixed $d$. Unlike the methods in \cite{meng2026computation}, which are tailored to the quadratic loss in the special case of the penalized trimmed least squares problem, the CORe conditions are derived from the general coordinate-optimality principle of Section~\ref{sec:cor_framework} and apply to any loss and link satisfying Assumption~\ref{assump:loss_cor}, including the logistic model of Section~\ref{subsec:classification}. Proposition~\ref{prop:poly_sim} below extends the resulting polynomiality guarantees from least trimmed squares to the general robust single-index class.

\subsection{Robust Logistic Regression}\label{subsec:classification}
As a special case, we consider robust binary classification via logistic
regression. Here the response $r \in \{0,1\}^n$ collects binary labels, the
link function is the sigmoid
\[
    f(t) := \frac{1}{1+e^{-t}},
\]
and the loss is binary cross-entropy
\[
    \ell(r,p) := -r\log p - (1-r)\log(1-p), \qquad p \in (0,1).
\]
To apply the CORe formulation~\eqref{eq:sim_core_reformulation}, we
first characterize
\[
    S_i(\lambda_i) := \{t \in \mathbb{R} : \ell_i(f(t)) \le \lambda_i\}.
\]
If $r_i = 1$, then $\ell_i(f(t)) = -\log(f(t))$. This leads to
\[
    \ell_i(f(t)) \le \lambda_i
    \;\Longleftrightarrow\;
    t \ge -\log(e^{\lambda_i}-1),
\]
giving $S_i(\lambda_i) = [-\log(e^{\lambda_i}-1),\, +\infty)$. Similarly, if $r_i = 0$,
then $\ell_i(f(t)) = -\log(1-f(t))$. This yields
\[
    \ell_i(f(t)) \le \lambda_i
    \;\Longleftrightarrow\;
    t \le \log(e^{\lambda_i}-1),
\]
giving $S_i(\lambda_i) = (-\infty,\, \log(e^{\lambda_i}-1)]$. Both cases
unify compactly as
\[
    (2r_i-1)\,t \ge \tau_i, \qquad \tau_i := -\log(e^{\lambda_i}-1).
\]
Note that, in this subsection, $\tau_i$ denotes this logistic threshold, which may take either sign, in contrast with the nonnegative thresholds of Section~\ref{sec:applications_QP}. Hence each coordinate induces a two-way split,
\begin{align*}
    \mathcal{D}_{i, \mathrm{OC}}^0 &:= \left\{(\theta,w,z):\; z_i=0,\; w_i=0,\; (2r_i-1)\langle \phi_i,\theta\rangle\ge \tau_i\right\},\\
    \mathcal{D}_{i, \mathrm{OC}}^1 &:= \left\{(\theta,w,z):\; z_i=1,\; (2r_i-1)\langle \phi_i,\theta\rangle \le \tau_i\right\}.
\end{align*}

For the logistic model, the loss admits the compact representation $\ell_i(f(t)) = \log\bigl(1 + e^{-(2r_i-1)t}\bigr)$. Introduce disaggregated variables $t_i^0, t_i^1$ with $(2r_i-1)\langle \phi_i,\theta\rangle = t_i^0 + t_i^1$, where $t_i^1$ is the on-regime copy. Eliminating the correction variable $w_i$---whose only role at optimality is to nullify the loss of an outlying observation---we charge observation $i$ the loss $\log(1+e^{-t_i^0})$ evaluated at the off-regime copy. When $z_i = 1$, the disjunctive bounds force $t_i^0 = 0$, so this term contributes the constant $\log 2$; correcting the outlier penalty to $\lambda_i - \log 2$ makes the total contribution exactly $\lambda_i$. (In particular, the resulting formulation remains valid for any $\lambda_i > 0$, even when the coefficient $\lambda_i - \log 2$ is negative.) This yields the following two-regime model.
\begin{equation} \label{eq:robust_logistic_core} 
\begin{aligned} \min_{\theta\in\mathbb{R}^d,\; t^0,t^1\in\mathbb{R}^n,\; z\in\{0,1\}^n}\quad & \sum_{i=1}^n \log\!\left(1+\exp(-t_i^0)\right) + \frac{1}{d}\|\theta\|_2^2 + \sum_{i=1}^n (\lambda_i-\log 2)z_i \\ \text{s.t.}\quad & (2r_i-1)\langle \phi_i,\theta\rangle = t_i^0+t_i^1, \quad i=1,\ldots,n,\\ & \tau_i(1-z_i)\le t_i^0\le U_i(1-z_i), \quad  i=1,\ldots,n,\\ & L_i z_i\le t_i^1\le \tau_i z_i, \quad  i=1,\ldots,n. 
\end{aligned} 
\end{equation}

Indeed, when $z_i=0$, we have $t_i^1=0$, so $t_i^0=(2r_i-1)\langle\phi_i, \theta\rangle$ and the objective includes the logistic loss. When $z_i=1$, it gives $t_i^0=0$, so the objective contribution is $\log 2+(\lambda_i-\log 2)=\lambda_i$,
which is exactly the cost of treating observation $i$ as an outlier.

Formulation \eqref{eq:robust_logistic_core} therefore encodes the coordinate-wise CORe logic for
robust binary classification. In contrast to the earlier example
based on bounded intervals, the present example yields a one-sided half-space condition, and hence a two-term disjunction for each observation.

Since $(\theta,z)=(0,0)$ is feasible for \eqref{eq:sim_original_cor} with objective value $n\log 2$, and the CORe reformulation preserves the optimal value, any optimal solution $\theta^\star$ satisfies
\[
\frac{1}{d}\|\theta^\star\|_2^2 \le n\log 2,
\]
and hence
\[
\|\theta^\star\|_2 \le \sqrt{nd\log 2}.
\]
It follows that
\[
\bigl|(2r_i-1)\phi_i^\top \theta^\star\bigr|
\le \|\phi_i\|_2\,\|\theta^\star\|_2
\le \|\phi_i\|_2\sqrt{nd\log 2}.
\]
Thus, a valid choice for the bounds on $(2r_i-1)\langle \phi_i,\theta\rangle $ is
\[
L_i := -\|\phi_i\|_2\sqrt{nd\log 2},
\qquad
U_i := \|\phi_i\|_2\sqrt{nd\log 2},
\qquad i=1,\dots,n.
\]
As we show next in Section~\ref{subsec:sim_poly}, the convexity of the logistic loss further implies that robust logistic regression can be solved in polynomial time when the number of covariates $d$ is fixed.

\subsection{Polynomial Complexity for Fixed Dimension}
\label{subsec:sim_poly}

We now show that the polynomial bounds established for the unbalanced bipartite quadratic model in Proposition~\ref{prop:poly_tree} are not an artifact of the quadratic objective: they extend to the entire robust single-index class, covering in particular the two special cases discussed above. The driving structure is the same. The coordinate-wise optimality conditions constrain $\theta$ only through the $n$ scalars $\langle \phi_i,\theta\rangle$, and the finite endpoints of the intervals $S_i(\lambda_i)$ define at most $2n$ hyperplanes
\[
\langle \phi_i,\theta\rangle=a_i,
\qquad
\langle \phi_i,\theta\rangle=b_i,
\qquad i=1,\ldots,n,
\]
in $\mathbb{R}^d$. Within any cell of the arrangement of these hyperplanes, the position of every index $\langle \phi_i,\theta\rangle$ relative to $S_i(\lambda_i)$ is fixed, and hence so is the regime of every observation.

\begin{proposition}[Polynomial bounds for robust single-index models]
\label{prop:poly_sim}
Suppose Assumption~\ref{assump:loss_cor} holds and the number of covariates $d$ is fixed.
\begin{enumerate}
\item[(i)] At most $O(n^d)$ regime patterns are consistent with the coordinate-wise optimality conditions. Consequently, a branch-and-bound scheme applied to formulation \eqref{eq:sim_core_reformulation} that branches on the regime variables $z_i^-, z_i^+$ and prunes infeasible nodes explores $O(n^{d+1})$ nodes.
\item[(ii)] If, in addition, the functions $t\mapsto \ell_i(f(t))$ are convex for all $i$, then problem \eqref{eq:sim_original_cor} can be solved by solving $O(n^{d})$ convex optimization problems, and hence in polynomial time.
\end{enumerate}
\end{proposition}

\begin{proof}
(i) The at most $2n$ hyperplanes $\langle \phi_i,\theta\rangle=a_i$ and $\langle \phi_i,\theta\rangle=b_i$ (only finite endpoints contribute) partition $\mathbb{R}^d$ into at most $\sum_{\ell=0}^{d}\binom{2n}{\ell}=O(n^d)$ cells \cite{zaslavsky1975facing}. Fix a cell $C$. For every $\theta\in C$, the signs of $\langle \phi_i,\theta\rangle-a_i$ and $\langle \phi_i,\theta\rangle-b_i$ are constant, so the regime of every observation dictated by the coordinate-wise optimality conditions is the same for all $\theta\in C$. Each cell is therefore consistent with exactly one regime pattern, and every pattern consistent with the coordinate-wise optimality conditions arises from at least one cell. This proves the $O(n^d)$ bound on the number of patterns.

For the node bound, observe that all regime constraints in \eqref{eq:sim_core_reformulation} are linear in $(\theta,w,t,t^0,t^-,t^+,z,z^-,z^+)$. As in the proof of Proposition~\ref{prop:poly_tree}, the nodes with feasible relaxations form a subtree containing the root; its leaves correspond to distinct consistent regime patterns, so the subtree has at most $O(n^d)$ leaves and, having depth at most $2n$, at most $O(n\cdot n^{d})$ nodes. Every pruned node is a child of a node of this subtree, and each node has at most two children, so the total number of nodes is $O(n^{d+1})$.

(ii) Eliminating $w$ and $z$ coordinate-wise, as in the derivation of the on- and off-regime values, problem \eqref{eq:sim_original_cor} is equivalent to
\[
\min_{\theta\in\mathbb{R}^d}\;
\sum_{i=1}^n \min\Bigl\{\ell_i\bigl(f(\langle \phi_i,\theta\rangle)\bigr),\,\lambda_i\Bigr\}
+\frac{1}{d}\|\theta\|_2^2 .
\]
On the closure $\bar C$ of any cell $C$, each term $\min\{\ell_i(f(\langle \phi_i,\theta\rangle)),\lambda_i\}$ coincides with a single branch: with $\ell_i(f(\langle \phi_i,\theta\rangle))$ if $\langle \phi_i,\theta\rangle\in S_i(\lambda_i)$ on $C$, and with the constant $\lambda_i$ otherwise. (On the boundary of $\bar C$, the two branches agree, since, there, $\ell_i(f(\langle \phi_i,\theta\rangle))=\lambda_i$.) Under the convexity hypothesis, the restricted objective is therefore convex on the polyhedron $\bar C$. Since the closure of some cell contains a global minimizer, it suffices to minimize the restricted objective over $\bar C$ for each of the $O(n^d)$ cells and return the best solution found. For fixed $d$, the cells and their associated regime patterns can be enumerated in polynomial time, e.g., by the facet-parametrization procedure of Appendix~\ref{sec:poly-algo}. \qed
\end{proof}

The convexity hypothesis in part (ii) is satisfied, in particular, by any convex loss combined with an affine link---covering penalized least trimmed squares---and by the logistic model of Section~\ref{subsec:classification}, for which $\ell_i(f(t))=\log\bigl(1+e^{-(2r_i-1)t}\bigr)$ is convex. For least trimmed squares, Proposition~\ref{prop:poly_sim} recovers guarantees of the type established in \cite{meng2026computation}; the proposition shows that they persist far beyond the quadratic loss. In contrast, without the convexity hypothesis, part (i) still bounds the combinatorial search, but the node relaxations of \eqref{eq:sim_core_reformulation} need not be tractable, in line with the limitations discussed in Remark~\ref{remark:CF-single-index}.

\paragraph{Computational tradeoffs.}
Before turning to the experiments, we emphasize that CORe introduces a tradeoff. On the one hand, the added threshold, stationarity, and disjunctive constraints make the regime logic explicit, which can help general-purpose solvers through presolve, bound tightening, variable fixing, and cut generation, as well as through more informative branching. On the other hand, they also enlarge the model, and the resulting relaxations can be more expensive to solve. Thus, CORe is most beneficial when the coordinate-wise conditions are informative, as in sparse or low-dimensional-control settings, but its overhead may dominate when the coordinate-wise conditions are uninformative or when each coordinate is coupled with many others.

\section{Experiments}
\label{sec:experiments}

\subsection{Experimental Setup}

All experiments are implemented in Python 3.9 and solved using Gurobi 12.0.3 through its Python API. The computations are conducted on a single compute node equipped with an Intel Xeon Platinum 8592+ processor, using 8 CPU cores and 12 GB of RAM. Unless otherwise specified, Gurobi parameters are kept at their default values. Time limits are set to 3600 seconds for most instances and 600 seconds for structured instances such as path, star, and tree graphs. All reported run times correspond to solver time and exclude model construction overhead. The results reported in the tables and figures are averaged over five independent trials.

\subsection{Quadratic Programs with Indicators}
\label{sec:qp_experiments}

We evaluate the computational performance of the proposed formulation on a family of synthetic quadratic programs with indicator variables of the form
\[
\min_{x \in \mathbb{R}^n,\, z \in \{0,1\}^n}
\frac{1}{2}(d - x)^\top Q (d - x) + \lambda^\top z
\quad \text{s.t.} \quad x_i(1 - z_i) = 0,\; i=1,\dots,n.
\]

To systematically study the effect of sparsity, we generate $Q$ from a weighted Erd\H{o}s--R\'enyi random graph. 
For each pair $(i,j)$ with $i<j$, an edge is included independently with probability $\delta$; whenever an edge is present, its weight $W_{ij}$ is drawn independently from a uniform distribution on $[0.1,1]$, and we set $W_{ji}=W_{ij}$ to ensure symmetry, while $W_{ij}=0$ if no edge is present. The weighted graph Laplacian $L$ is constructed as
\[
L = D - W,
\]
where $W$ is the weighted adjacency matrix and $$D=\mathrm{diag}(\sum_{j=1}^nW_{1j},\ldots,\sum_{j=1}^nW_{nj})$$ is the diagonal degree matrix. 
The Laplacian is positive semidefinite by construction. 
To control conditioning, we define
\[
Q = L + \alpha I,
\]
where $\alpha$ is chosen such that the condition number satisfies $\kappa(Q) = 100$. 
Finally, $Q$ is scaled so that $\operatorname{tr}(Q)/n = 1$, ensuring comparable magnitude across different sparsity levels. 
The vector $d \in \mathbb{R}^n$ is generated independently from a standard Gaussian distribution and then normalized to satisfy $\|d\|_2 = 1.$
 
All instances are solved using identical solver settings. We report the following
metrics. ``Root gap'' is the relative optimality gap at the root node. ``End UB'' and ``End LB'' are the final upper and lower bounds returned by the solver, and ``End gap'' is the final relative optimality gap. In tables where the compared formulations attain the same final upper bound up to numerical tolerance, we omit the ``End UB'' and ``End LB'' columns and report only the ``End gap,'' since the absolute bound values provide little additional information for the computational comparison. The column ``Active (\%)'' reports \(100\|z\|_0/n\), the percentage of selected coordinates. ``Nodes'' is the number of branch-and-bound nodes explored, and
``Time (s)'' is the solution time in seconds. The notation ``TL'' indicates that
the solver reached the time limit.

For all experiments involving big-$M$ formulations, a common heuristic to choose $M_i$ is to initialize the parameter $M_i$ with a conservative value and then tighten it using the solution of a continuous relaxation \cite{park2017bayesian,manzour2021integer,kucukyavuz2023consistent,xu2025integer}. Specifically, if $x^{\mathrm{relax}}$ denotes the relaxation solution, we update
\[
M \leftarrow \min\Bigl\{M,\; 2\max_i |x_i^{\mathrm{relax}}|\Bigr\}.
\]
This tightening is applied uniformly across all experiments as a practical way to reduce the weakness of the big-$M$ constraints.

\paragraph{Problem size $n$.}

\begin{table}[t]
\centering
\caption{Average computational results for QP with indicators with fixed graph sparsity \(\delta=0.01\).}
\label{tab:qp_problem_size_delta001_full}
\resizebox{\columnwidth}{!}{\begin{tabular}{ccrrrrr}
\toprule
$n$ & Formulation & Root gap & End gap & Active (\%) & Nodes & Time (s) \\
\midrule
100  & CORe     & 0.1\%& 0& 13.6\%& 8& 0.08
\\
100  & Original & 19.5\%& 0& 13.6\%& 25595& 0.70
\\
\midrule
200  & CORe     & 1.8\%& 0& 13.6\%& 152& 0.16
\\
200  & Original & 58.3\%& 29.2\%& 13.4\%& 48282010& TL\\
\midrule
300  & CORe     & 1.2\%& 0& 14.1\%& 127& 0.18
\\
300  & Original & 60.6\%& 42.2\%& 13.8\%& 24356213& TL\\
\midrule
500  & CORe     & 0.5\%& 0& 11.7\%& 84.8& 0.43
\\
500  & Original & 63.6\%& 52.2\%& 11.7\%& 10426955& TL\\
\midrule
800  & CORe     & 0.1\%& 0& 10.1\%& 17.6& 3.10
\\
800  & Original & 62.1\%& 56.9\%& 10.1\%& 2639442& TL\\
\midrule
1000 & CORe     & 40.0\%& 0& 9.9\%& 1905& 429.63
\\
1000 & Original & 64.2\%& 60.7\%& 10.0\%& 1207135& TL\\
\bottomrule
\end{tabular}}
\end{table}

Table~\ref{tab:qp_problem_size_delta001_full} reports the performance of the CORe and original formulations for QP with indicators under a sparse graph structure (\(\delta=0.01\)) as the problem size \(n\) increases. The results demonstrate a clear and consistent advantage of the CORe formulation across all instance sizes. For small instances (\(n=100\)), both formulations solve the problem to optimality, although CORe already requires several orders of magnitude fewer branch-and-bound nodes. As \(n\) increases, the gap between the two formulations becomes dramatic. The CORe formulation maintains near-zero root gaps and consistently solves all instances to optimality, with solution times remaining below one second up to \(n=500\) and only modestly increasing for larger instances. In contrast, the original formulation exhibits large root gaps (above \(58\%\)) and fails to close the optimality gap for all instances with \(n \geq 200\), hitting the time limit in every such case and leaving substantial final gaps.

The difference in computational effort is further reflected in the sizes of the branch-and-bound trees. CORe requires only tens to a few thousand nodes for the largest instance, whereas the original formulation explores millions to tens of millions of nodes without reaching optimality. Notably, the number of nodes required by CORe remains extremely small up to \(n=800\), suggesting that the enhanced formulation effectively prunes most suboptimal choices of the indicator variables, leaving only a few candidates that require branching and ultimately leading to far fewer nodes in the branch-and-bound tree. Overall, these results indicate that for sparse interaction structures, the CORe formulation is more scalable than the standard big-$M$ formulation.

These gains appear to stem not only from a stronger lower bound but also from the CORe formulation being much more informative to the solver during presolve. In particular, the added threshold and stationarity relations explicitly expose the regime logic linking the continuous variables and the binary indicators, allowing presolve to tighten bounds, fix variables, eliminate redundant constraints, and simplify large portions of the model before branching. The same structural information can also support stronger implied inequalities and cut generation during the search. This effect is especially pronounced in sparse instances, where the local coordinate-wise conditions are highly informative, and helps explain why CORe often solves near the root even when the standard big-$M$ formulation still requires exploring a very large branch-and-bound tree.

\paragraph{Graph sparsity level $\delta$.}

\begin{table}[t]
\centering
\caption{Average computational results for QP with indicators with fixed problem size and varying graph sparsity $\delta$.}
\label{tab:qp_graph_sparsity}
\resizebox{\columnwidth}{!}{\begin{tabular}{ccrrrrr}
\toprule
$\delta$ & Formulation & Root gap & End gap & Active (\%) & Nodes & Time (s) \\
\midrule
0.01 & CORe     & 0& 0& 5.04\% & 1        & 1.79 \\
0.01 & Original & 53.0\%& 40.1\%& 5.04\% & 8556024  & TL \\
\midrule
0.02 & CORe     & 0& 0& 4.48\% & 1        & 1.08 \\
0.02 & Original & 50.5\%& 38.6\%& 4.44\% & 7952371  & TL \\
\midrule
0.05 & CORe     & 17.9\%& 0& 4.04\% & 6203     & 166.04 \\
0.05 & Original & 51.9\%& 40.4\%& 4.04\% & 5327534  & TL \\
\midrule
0.10 & CORe     & 21.7\%& 0& 3.48\% & 8897     & 111.82 \\
0.10 & Original & 52.1\%& 41.8\%& 3.48\% & 3004007  & TL \\
\midrule
0.20 & CORe     & 41.4\%& 0& 4.24\% & 15604    & 459.59 \\
0.20 & Original & 53.7\%& 42.7\%& 4.24\% & 2144167  & TL \\
\midrule
0.50 & CORe     & 51.6\%& 0& 3.36\% & 20942    & 691.38 \\
0.50 & Original & 54.1\%& 43.9\%& 3.36\% & 1383747  & TL \\
\bottomrule
\end{tabular}}
\end{table}

Table~\ref{tab:qp_graph_sparsity} reports the performance of the CORe and original formulations as the graph sparsity parameter \(\delta\) varies while the problem size is fixed at $n=500$. The results reveal a clear dependence of the CORe formulation on the sparsity structure of the underlying graph. When the graph is highly sparse (\(\delta = 0.01, 0.02\)), CORe exhibits extremely strong performance, solving all instances to optimality almost immediately with a single branch-and-bound node and negligible computation time. In contrast, the original formulation exhibits large root gaps (around \(50\%\)) and fails to close the optimality gap within the time limit, even after exploring millions of nodes.

As the graph becomes denser, the performance of CORe gradually deteriorates. While it continues to solve all instances to optimality, the root gap increases significantly (from near zero to above \(50\%\)), and both the number of nodes and solution time grow substantially. For example, at \(\delta = 0.50\), CORe requires over \(20{,}000\) nodes and several hundred seconds to solve the problem. Nevertheless, it consistently outperforms the original formulation, which remains unable to close the gap at all sparsity levels and continues to reach the time limit, resulting in large final gaps (around \(40\%\)).

These results highlight an important structural property of the CORe formulation: its effectiveness is strongly tied to sparsity. In highly sparse settings, the coordinate-wise optimality conditions nearly determine the global structure, leading to extremely small branch-and-bound trees. As the graph becomes denser, interactions between coordinates weaken the effectiveness of local conditions, leading to larger branch-and-bound trees and increased computational effort. However, even in these more challenging regimes, CORe maintains a significant advantage over the standard big-$M$ formulation.

\paragraph{Sparsity penalty level $\lambda$.}

\begin{table}[t]
\centering
\caption{Average computational results for QP with indicators under varying penalty parameter $c_1$.}
\label{tab:qp_tau_sparsity}
\resizebox{\columnwidth}{!}{\begin{tabular}{ccrrrrr}
\toprule
$c_1$ & Formulation & Root gap & End gap & Active (\%) & Nodes & Time (s) \\
\midrule
0.01 & CORe     & 100.0\%& 5.3\%& 43.48\% & 1646377  & TL \\
0.01 & Original & 79.8\%& 66.4\%& 43.72\% & 7507499  & TL \\
\midrule
0.02 & CORe     & 27.6\%& 1.6\%& 29.00\% & 6185927  & TL \\
0.02 & Original & 75.8\%& 63.6\%& 29.04\% & 10060650 & TL \\
\midrule
0.05 & CORe     & 0.5\%& 0& 13.04\% & 149      & 2.49 \\
0.05 & Original & 64.6\%& 54.1\%& 12.96\% & 8012889  & TL \\
\midrule
0.08 & CORe     & 0& 0& 7.08\%  & 1        & 0.20 \\
0.08 & Original & 57.6\%& 45.3\%& 7.00\%  & 10523945 & TL \\
\midrule
0.12 & CORe     & 0& 0& 4.00\%  & 1        & 1.20 \\
0.12 & Original & 49.2\%& 35.2\%& 4.00\%  & 9643881  & TL \\
\bottomrule
\end{tabular}}
\end{table}

Table~\ref{tab:qp_tau_sparsity} reports the performance of the CORe and original formulations as the parameter \(c_1\), which controls the sparsity parameter $\lambda$, varies where \(\lambda = c_1 \log(n)/n\), following a standard high-dimensional model-selection scaling that balances the \(O(1/n)\) per-coordinate signal size with a \(\log(n)\) correction for selecting among \(n\) variables. Here, the problem size is fixed at $n = 500$. As \(c_1\) increases, the effective penalty becomes stronger, which encourages sparser solutions; this is reflected in the decrease of the number of nonzeros from about \(218\) at \(c_1=0.01\) to \(20\) at \(c_1=0.12\). The results show that the performance of CORe improves markedly as the penalty increases. For very small values of \(c_1\), where the penalty is weak and the resulting solution is relatively dense, the problem remains challenging even for CORe: at \(c_1=0.01\), CORe reaches the time limit, but still substantially outperforms the original formulation by attaining a much smaller final gap (\(5.3\%\) versus \(66.4\%\)) and exploring far fewer nodes. At \(c_1=0.02\), CORe nearly solves the problem to optimality, reducing the final gap to \(1.6\%\), while the original formulation again times out with a large gap.

As \(c_1\) increases, the penalty parameter \(\lambda=c_1\log(n)/n\) becomes larger,
which encourages sparser optimal solutions. This tends to make the problem easier for
both formulations, but the effect is much more pronounced for CORe because the
penalty also appears in the coordinate-wise threshold
\(\tau_i=\sqrt{2\lambda_i Q_{ii}}\). Larger values of \(\lambda_i\) enlarge the
off-regime region \(|g_i|\le \tau_i\), allowing more coordinates to be certified as
inactive through the CORe conditions. This effect is reflected in
Table~\ref{tab:qp_tau_sparsity}. Once \(c_1\) becomes moderate, the advantage of
CORe becomes dominant. For \(c_1 \geq 0.05\), CORe solves all instances to
optimality, and its computational effort drops sharply: the node count decreases from
millions to only \(149\) nodes at \(c_1=0.05\), and to a single node at
\(c_1=0.08\) and \(c_1=0.12\). In contrast, the original formulation continues to
exhibit weak root relaxations, explores millions of nodes in every setting, and fails
to close the gap within the time limit. Overall, these results suggest that CORe is
particularly effective when many coordinates can be ruled out early in the
branch-and-bound process.

\paragraph{Two-dimensional Markov random fields.}

To further evaluate the proposed reformulation on a structured sparse quadratic model beyond random graphs considered earlier, we consider a two-dimensional Gaussian Markov random field (2D-GMRF) image denoising problem following the experimental setup of \cite{he2024comparing}. Here, the term ``image'' refers to a \(p\times p\) array of signal values: each entry \(X_{ij}\) can be interpreted as the intensity of a pixel, and the goal is to recover the latent signal \(X\) from noisy observations. The grid structure of the resulting sparsity graph induces spatial dependence, since neighboring pixels are expected to have similar values.

Specifically, we generate the latent signal \(X\in\mathbb{R}^{p\times p}\) as follows. Let \(s\) denote the spike size and \(h\) denote the number of Gaussian shocks. We initialize \(X=0\) and then add \(h\) local Gaussian shocks. For each shock, we choose the upper-left corner \((i_0,j_0)\) of an \(s\times s\) block uniformly at random from
\[
\{1,\ldots,p-s+1\}\times \{1,\ldots,p-s+1\}.
\]
We then generate \(W\in\mathbb{R}^{s\times s}\) with $\operatorname{vec}(W)\sim N(0,\Omega_s^{-1}),$
where \(\Omega_s\in\mathbb{R}^{s^2\times s^2}\) is defined on the \(s\times s\) grid by
$(\Omega_s)_{uv,uv}=4$ for every $u,v\in [s]$, and $
(\Omega_s)_{uv,u'v'}=-1 $ if $(u,v)$ and  $(u',v')$ are horizontally or vertically adjacent. All other entries of $\Omega$ are set to zero. The shock is added to the selected block:
\[
X_{i_0:i_0+s-1,\;j_0:j_0+s-1}
\leftarrow
X_{i_0:i_0+s-1,\;j_0:j_0+s-1}+W.
\]
The resulting signal is sparse at the pixel level, with nonzero values concentrated in a small number of contiguous \(s\times s\) blocks. Finally, we observe
\[
\widetilde X_{ij}=X_{ij}+\epsilon_{ij},\qquad
\epsilon_{ij}\overset{\mathrm{i.i.d.}}{\sim}N(0,\sigma^2).
\]
Given the noisy observations \(\widetilde X\), we estimate the latent signal by solving the exact \(\ell_0\)-regularized GMRF recovery problem
\[
\begin{aligned}
    \min_{X\in\mathbb{R}^{p\times p}} \;
\sum_{i=1}^{p}\sum_{j=1}^{p} \frac{1}{\sigma^2}(\widetilde{X}_{ij}-X_{ij})^2
+ \sum_{i=1}^{p-1}\sum_{j=1}^{p} (X_{ij}-X_{i+1,j})^2
+\\ \sum_{i=1}^{p}\sum_{j=1}^{p-1} (X_{ij}-X_{i,j+1})^2
+ \lambda \|X\|_0 .
\end{aligned}
\]
The first term measures fidelity to the observed noisy data. The second and third terms are the GMRF smoothness penalties: they penalize differences between vertically and horizontally adjacent pixels, respectively, thereby encouraging neighboring pixels to have similar values. The last term penalizes the number of nonzero pixels and reflects the prior that the true signal is spatially localized and sparse.

Let \(n=p^2\), and vectorize the matrices $X$ and $\widetilde X$ as
\[
x=\operatorname{vec}(X)\in\mathbb{R}^n,\qquad
\tilde x=\operatorname{vec}(\widetilde X)\in\mathbb{R}^n.
\]
Introducing binary variables \(z\in\{0,1\}^n\) to encode the support of \(x\), and dropping the additive constant term, the estimator can be written in the standard quadratic form
\[
\min_{x,z} \; \frac{1}{2}x^\top Qx + c^\top x + \lambda\mathbf{1}^\top z
\quad \text{s.t.} \quad
x_i(1-z_i)=0,\; i=1,\ldots,n .
\]
Here, \(Q\) and \(c\) are obtained by collecting the quadratic and linear terms in
\[
\frac{1}{\sigma^2}\|\tilde x-x\|_2^2
+
\sum_{i=1}^{p-1}\sum_{j=1}^{p} (X_{ij}-X_{i+1,j})^2
+
\sum_{i=1}^{p}\sum_{j=1}^{p-1} (X_{ij}-X_{i,j+1})^2 .
\]
In particular, $c=-\frac{2}{\sigma^2}\tilde x.$
The matrix \(Q\) is sparse because each pixel is coupled only with its horizontal and vertical neighbors. It is also positive definite, since the squared-error term contributes a positive diagonal quadratic term and the neighboring-difference penalties are positive semidefinite.

This instance class is useful for evaluating CORe beyond random sparse problems. Although the quadratic matrix is sparse due to local pixel interactions, its sparsity pattern is not random but follows a two-dimensional grid. 

In the experiments in Table~\ref{tab:2dmrf_results}, we compare the original indicator formulation with the CORe reformulation on this instance class. We evaluate statistical recovery using the true positive rate (TPR), false positive rate (FPR), and $F_1$-score, computed with respect to the true support of the signal. We set the $\ell_0$ regularization parameter (equivalent to the sparsity penalty level) to $\lambda = 0.5 \log(n).$
The constant \(0.5\) was chosen from the grid \(\{0.1,0.2,\ldots,1.0\}\) based on the overall $F_1$ score on the generated 2D-MRF instances. Since statistical recovery is not the primary focus of this experiment, this simple tuning rule is sufficient for producing representative sparse recovery instances for comparing the formulations.

\begin{table}[t]
\centering
\caption{Average computational results for the 2D-MRF instances under different noise levels.}
\label{tab:2dmrf_results}
\resizebox{\columnwidth}{!}{\begin{tabular}{cccrrrrrrr}
\toprule
$n$ & $\sigma^2$ & Formulation & End gap & TPR & FPR & $F_1$ & Active (\%) & Nodes & Time (s) \\
\midrule
100 & 0.01 & CORe     & 0& 0.776 & 0.134 & 0.629 & 24.0\%& 0.8 & 0.02 \\
100 & 0.01 & Original & 0& 0.776 & 0.134 & 0.629 & 24.0\%& 5.4 & 0.28 \\
100 & 0.02 & CORe     & 0& 0.701 & 0.130 & 0.599 & 22.6\%& 1.0 & 0.02 \\
100 & 0.02 & Original & 0& 0.701 & 0.130 & 0.599 & 22.6\%& 2.8 & 0.24 \\
100 & 0.05 & CORe     & 0& 0.633 & 0.088 & 0.623 & 18.6\%& 0.6 & 0.02 \\
100 & 0.05 & Original & 0& 0.633 & 0.088 & 0.623 & 18.6\%& 2.6 & 0.40 \\
\midrule
400 & 0.01 & CORe     & 0& 0.755 & 0.085 & 0.630 & 16.3\%& 1.0 & 0.04 \\
400 & 0.01 & Original & 0& 0.755 & 0.085 & 0.630 & 16.3\%& 34.4 & 1.09 \\
400 & 0.02 & CORe     & 0& 0.689 & 0.073 & 0.616 & 14.5\%& 1.0 & 0.04 \\
400 & 0.02 & Original & 0& 0.689 & 0.073 & 0.616 & 14.5\%& 64.0 & 1.08 \\
400 & 0.05 & CORe     & 0& 0.569 & 0.055 & 0.555 & 10.8\%& 1.0 & 0.04 \\
400 & 0.05 & Original & 0& 0.569 & 0.055 & 0.555 & 10.8\%& 40.8 & 1.57 \\
\midrule
1600 & 0.01 & CORe     & 0& 0.770 & 0.053 & 0.641 & 10.8\%& 1.0 & 0.14 \\
1600 & 0.01 & Original & 0& 0.770 & 0.053 & 0.641 & 10.8\%& 157.4 & 5.79 \\
1600 & 0.02 & CORe     & 0& 0.639 & 0.046 & 0.583 & 9.1\%& 1.0 & 0.14 \\
1600 & 0.02 & Original & 0& 0.639 & 0.046 & 0.583 & 9.1\%& 917.2 & 8.89 \\
1600 & 0.05 & CORe     & 0& 0.487 & 0.040 & 0.491 & 7.3\%& 1.0 & 0.17 \\
1600 & 0.05 & Original & 0& 0.487 & 0.040 & 0.491 & 7.3\%& 2012860.6 & 1982 \\
\bottomrule
\end{tabular}}
\end{table}

Table~\ref{tab:2dmrf_results} shows that the two formulations return the same statistical recovery performance and essentially the same optimal objective values, as expected, since they model the same \(\ell_0\)-regularized estimator. In all tested instances, both formulations solve to zero optimality gap, and the TPR, FPR, \(F_1\)-score, and active fraction are identical across the two formulations. The main difference is computational. CORe solves all instances essentially at the root node, with an average node count close to one and runtimes below \(0.2\) seconds even when \(n=1600\). In contrast, the original formulation requires substantially more branching, and its performance deteriorates as the problem size increases. For example, when \(n=1600\) and \(\sigma^2=0.05\), the original formulation explores over two million nodes and takes \(1982\) seconds, whereas CORe solves the same instances in \(0.17\) seconds on average. 
As $\sigma^2$ increases, the data-fidelity term contributes a smaller diagonal component to $Q$, making the quadratic objective less diagonally dominant relative to the coupling terms. This weakens the information available to the standard big-M relaxation. In contrast, CORe explicitly encodes the coordinate-wise threshold rule, so the solver can still infer many support decisions directly from the local optimality conditions. These results indicate that the coordinate-wise optimality structure exploited by CORe remains highly effective on two-dimensional grid instances, showcasing its applicability beyond random instances.

\subsection{Structural Sensitivity}

In this section, we study how the sparsity structure of $Q$ affects the performance of the CORe framework.

\subsubsection{Star graphs}

We generate synthetic instances using a star-graph structure. For a given problem size \(n\), we construct a sparse matrix \(Q \in \mathbb{R}^{n\times n}\) whose sparsity pattern corresponds to the star graph \(S_n\), with node \(1\) as the center and nodes \(2,\ldots,n\) as leaves. The data generation process is the same as in Section \ref{sec:qp_experiments} except for the weight matrix: for each edge connecting the center node to a leaf, we draw an independent weight
$$
W_{1j} \sim \mathrm{Unif}(0.1,1),
\quad W_{j1}=W_{1j}, \qquad j=2,\ldots,n,
$$
and set all other off-diagonal entries to zero. Using these weights, we construct the corresponding weighted star-graph Laplacian and then normalize the resulting matrix \(Q\) as described in Section \ref{sec:qp_experiments}.

\begin{table}[t]
\centering
\caption{Average computational results for QP with indicators on star graphs as the problem size $n$ varies.}
\label{tab:star_problem_size}
\resizebox{\columnwidth}{!}{\begin{tabular}{ccrrrrr}
\toprule
$n$ & Formulation & Root gap & End gap & Active (\%) & Nodes & Time (s) \\
\midrule
100   & CORe        & 0       & 0        & 10.80\%& 0.2      & 0.02 \\
100   & parametric  & -       & 0        & 10.80\%& -        & 0.07 \\
100   & Original    & 14.0\%& 0        & 10.80\%& 1282.8   & 0.34 \\
\midrule
200   & CORe        & 0       & 0        & 8.30\%& 0.4      & 0.04 \\
200   & parametric  & -       & 0        & 8.30\%& -        & 0.13 \\
200   & Original    & 2.4\%& 0        & 8.30\%& 1065.6   & 0.54 \\
\midrule
300   & CORe        & 0       & 0        & 9.00\%& 0.0      & 0.05 \\
300   & parametric  & -       & 0        & 9.00\%& -        & 0.24 \\
300   & Original    & 4.5\%& 0        & 9.00\%& 730.0    & 0.71 \\
\midrule
500   & CORe        & 0       & 0        & 9.32\%& 0.0      & 0.08 \\
500   & parametric  & -       & 0        & 9.32\%& -        & 0.53 \\
500   & Original    & 1.6\%& 0        & 9.32\%& 7542.0   & 3.05 \\
\midrule
1000  & CORe        & 0       & 0        & 8.62\%& 0.2      & 0.19 \\
1000  & parametric  & -       & 0        & 8.62\%& -        & 1.88 \\
1000  & Original    & 3.8\%& 0.01   & 8.60\%& 8063.8   & 6.57 \\
\midrule
2000  & CORe        & 0       & 0        & 7.47\%& 0.0      & 0.40 \\
2000  & parametric  & -       & 0        & 7.47\%& -        & 7.95 \\
2000  & Original    & 1.4\%& 0.07   & 7.47\%& 307635.8 & 11.08 \\
\midrule
5000  & CORe        & 0       & 0        & 6.22\%& 1.0      & 1.91 \\
5000  & parametric  & -       & 0        & 6.22\%& -        & 42.93 \\
5000  & Original    & 0.6\%& 0.01   & 6.22\%& 1178.8   & 41.93 \\
\midrule
10000 & CORe        & 0       & 0        & 5.48\%& 1.8      & 7.31 \\
10000 & parametric  & -       & 0        & 5.48\%& -        & 179.03 \\
10000 & Original    & 0.5\%& 0.02   & 5.48\%& 16959.7  & TL \\
\bottomrule
\end{tabular}}
\end{table}

Table~\ref{tab:star_problem_size} reports the computational performance on star-graph instances. Across all problem sizes, the CORe formulation solves the instances with zero optimality gap and only a negligible number of branch-and-bound nodes. Its running time remains very small, increasing from 0.02 seconds for $n=100$ to 7.31 seconds for $n=10{,}000$. The specialized parametric algorithm of \cite{bhathena2025parametric} also solves all instances to optimality and produces the same objective values and active-set sizes as CORe. However, its running time grows more quickly with \(n\), reaching 179.03 seconds at \(n=10{,}000\), and is consistently slower than CORe for larger instances. The original big-M formulation also benefits from the simple star structure: although its root gap remains nonzero, it is relatively small for most larger instances, mostly below \(5\%\), and it closes the final gap to zero or near zero for all but the largest instance. Nevertheless, it still requires substantially more branching than CORe, with node counts reaching over \(3\times 10^5\) at \(n=2000\), and it hits the time limit at \(n=10{,}000\). Overall, these results show that star graphs are favorable to both the original formulation and the parametric algorithm, but CORe remains the most scalable approach by combining exactness with near-root solution behavior and very low running times.

\subsubsection{Path graphs}

We generate synthetic instances of quadratic programs with indicator variables using a path-graph structure. For a given problem size $n$, we construct a tridiagonal matrix $Q \in \mathbb{R}^{n \times n}$ whose sparsity pattern corresponds to the path graph $P_n$. 

The data generation process is the same as in Section \ref{sec:qp_experiments} except for the weight matrix: for each edge, we draw an independent weight
\[
W_{i,i+1} \sim \mathrm{Unif}(0.1,1),
\quad W_{i+1,i} = W_{i,i+1},
\]
and set all other off-diagonal entries to zero.

\begin{table}[t]
\centering
\caption{Average computational results for QP with indicators on path graphs as the problem size $n$ varies.}
\label{tab:path_problem_size}
\resizebox{\columnwidth}{!}{\begin{tabular}{ccrrrrr}
\toprule
$n$ & Formulation & Root gap & End gap & Active (\%) & Nodes & Time (s) \\
\midrule
100   & CORe     & 0& 0& 16.70\%  & 1        & 0.08\\
100   & parametric     & -& 0& 16.70\%  & -        & 0.08\\
100   & Original & 59.1\%& 13.2\%& 16.70\%  & 19939440 & TL \\
\midrule
200   & CORe     & 0& 0& 12.30\%  & 1        & 0.11 \\
200   & parametric     & -& 0& 12.30\%  & -        & 0.13\\
200   & Original & 62.2\%& 36.8\%& 12.30\%  & 12222982 & TL \\
\midrule
300   & CORe     & 0& 0& 10.77\%  & 1        & 0.12\\
300   & parametric     & -& 0& 10.77\%  & -        & 0.19\\
300   & Original & 58.8\%& 42.7\%& 10.77\%  & 8957234  & TL \\
\midrule
500   & CORe     & 0& 0& 8.80\%  & 1        & 0.15\\
500   & parametric     & -& 0& 8.80\%  & -        & 0.32\\
500   & Original & 56.9\%& 47.5\%& 8.78\%  & 5828057  & TL \\
\midrule
1000  & CORe     & 0& 0& 7.26\%  & 1        & 0.28\\
1000  & parametric     & -& 0& 7.26\%  & -        & 0.68\\
1000  & Original & 59.1\%& 54.5\%& 7.25\%  & 2598710  & TL \\
\midrule
2000  & CORe     & 0& 0& 5.60\% & 0.8      & 0.48\\
2000  & parametric     & -& 0& 5.60\% & -      & 1.27\\
2000  & Original & 50.7\%& 48.6\%& 5.58\% & 1191785  & TL \\
\midrule
5000  & CORe     & 0& 0& 4.41\% & 1        & 1.34\\
5000  & parametric     & -& 0& 4.41\% & -        & 3.60\\
5000  & Original & 47.8\%& 47.1\%& 4.41\% & 340304   & TL \\
\midrule
10000 & CORe     & 0& 0& 3.86\% & 1        & 2.44\\
10000 & parametric     & -& 0& 3.86\% & -        & 8.19\\
10000 & Original & 41.2\%& 40.9\%& 3.86\% & 139643   & TL \\
\bottomrule
\end{tabular}}
\end{table}

\begin{figure}[t]
    \centering
    \includegraphics[width=0.72\linewidth]{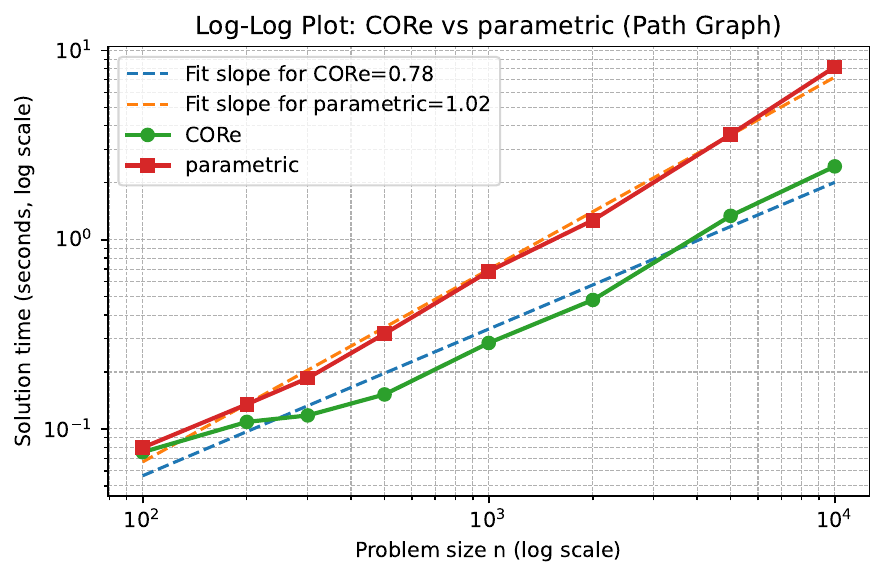}
    \caption{Log-log plot of solution time versus problem size for the CORe formulation and parametric algorithm on path-graph instances.}
    \label{fig:path_loglog}
\end{figure}

Table~\ref{tab:path_problem_size} reports the average computational performance of the original big-$M$ formulation, CORe, and the specialized parametric algorithm of \cite{bhathena2025parametric} on path-graph instances as the problem size $n$ increases from $100$ to $10{,}000$. We use the penalty parameter $\lambda = 0.2\log(n)/n.$
As in the previous experiments, this parameter is fixed to produce a nondegenerate sparse-solution regime rather than tuned for statistical performance. Across the tested sizes, the average active fraction decreases from \(16.70\%\) at \(n=100\) to \(3.86\%\) at \(n=10{,}000\).
It can be observed that CORe solves every instance to proven optimality essentially at the root node: the root gap is zero for all tested values of $n$, the average number of explored nodes is about one, and the runtime increases only mildly from $0.08$ seconds at $n=100$ to $2.44$ seconds at $n=10{,}000$. Moreover, our experiments are consistent with the observation of~\cite{bhathena2025parametric} that the parametric algorithm runs in almost linear time. Nevertheless, CORe consistently outperforms it, with the gap widening as $n$ increases: at $n = 10{,}000$, CORe requires $2.44$ seconds versus $8.19$ seconds for the parametric algorithm. The original big-M formulation, by contrast, performs poorly across all settings: it hits the time limit in every instance, exhibits root gaps of $41.2\%$--$62.2\%$, and terminates with substantial optimality gaps despite exploring many nodes. Overall, these results show that CORe is empirically competitive with---and on this benchmark faster than---a specialized exact algorithm designed for tree-structured problems.

Figure~\ref{fig:path_loglog} gives a complementary view of the path-graph experiments by plotting solution time against problem size on log-log scales. The approximately linear trends indicate stable polynomial-type empirical scaling for both methods over the tested range. In particular, the fitted slope is about $0.78$ for CORe and about $1.02$ for the parametric algorithm, suggesting that CORe exhibits a milder growth in runtime as $n$ increases, while the parametric algorithm scales close to linearly on these instances. Moreover, the CORe curve remains below the parametric curve throughout the entire range of problem sizes, and the gap becomes more visible for larger instances. 

We note that the remarkable performance of CORe on the random path instances does not conflict with Proposition~\ref{prop:path_counterexample_exp}. Rather, the proposition identifies a particular adversarial structure, in which the coefficients and thresholds are aligned so that many regime assignments on the path remain feasible. The random instances in Table~\ref{tab:path_problem_size} do not exhibit this alignment; instead, they lead to sparse solutions in which many inactive coordinates can be certified early by the CORe conditions. To examine this worst-case mechanism computationally, we next consider a family of path instances derived from the construction in Proposition~\ref{prop:path_counterexample_exp}.

\paragraph{Adversarial path graph instances.}
To better understand the worst-case behavior on path graphs, we generate a family of adversarial path instances motivated by Proposition~\ref{prop:path_counterexample_exp}. For each problem size \(n\), we set \(Q\) to be the tridiagonal path matrix with \(Q_{ii}=1\) and \(Q_{i,i+1}=Q_{i+1,i}=\rho\). The worst-case construction corresponds to the limiting choice \(\rho=1/2\). However, for a path matrix with \(Q_{ii}=1\) and constant off-diagonal value \(\rho\), the smallest eigenvalue is approximately \(1-2\rho\) for large \(n\). Thus, choosing \(\rho=1/2\) makes the matrix increasingly ill-conditioned as \(n\) grows. To avoid confounding the combinatorial difficulty with numerical ill-conditioning, we use \(\rho=0.45\), which preserves the same local path interaction structure while keeping \(Q\) uniformly positive definite.

This modification changes the threshold range needed for the worst-case mechanism. In the corresponding independent-set construction (see Appendix~\ref{sec:proof_counter} for more details), an inactive coordinate with exactly one active neighbor has \(|g_i|=1-\rho\). Therefore, the inactive CORe certificate \(|g_i|\le \tau\) requires \(\tau>1-\rho\). With \(\rho=0.45\), this lower boundary becomes \(1-\rho=0.55\). We therefore vary \(\tau\in\{0.6,0.7,0.8,0.9\}\), and set \(\lambda_i=\tau^2/2\) for all \(i\), since \(Q_{ii}=1\) and hence the CORe threshold is \(\tau_i=\sqrt{2\lambda_i}=\tau\). Under this choice, \(\tau=0.6\) is close to the boundary of the inactive certificate, while larger values of \(\tau\) provide increasingly wider certificate margins.

\begin{table}[t]
\centering
\caption{Average computational results for QP with indicators on adversarial path graphs as the threshold \(\tau\) and problem size \(n\) vary.}
\label{tab:path_adversarial}
\resizebox{\columnwidth}{!}{\begin{tabular}{cccrrrrr}
\toprule
\(\tau\) & \(n\) & Formulation & Root gap & End gap & Active (\%) & Nodes & Time (s) \\
\midrule
0.6 & 100 & CORe & 10.8\% & 1.9\% & 51.00\% & 16400931 & TL \\
0.6 & 100 & parametric & - & 0 & 51.00\% & - & 0.09 \\
0.6 & 100 & Original & 41.8\% & 17.9\% & 51.00\% & 21757580 & TL \\
\midrule
0.6 & 1000 & CORe & 14.6\% & 10.2\% & 50.10\% & 1896469 & TL \\
0.6 & 1000 & parametric & - & 0 & 50.10\% & - & 0.61 \\
0.6 & 1000 & Original & 89.8\% & 30.9\% & 50.10\% & 2380792 & TL \\
\midrule
0.7 & 100 & CORe & 0 & 0 & 50.00\% & 1 & 1.41 \\
0.7 & 100 & parametric & - & 0 & 50.00\% & - & 0.10 \\
0.7 & 100 & Original & 71.2\% & 31.5\% & 50.00\% & 21269267 & TL \\
\midrule
0.7 & 1000 & CORe & 0 & 0 & 50.00\% & 1 & 0.51 \\
0.7 & 1000 & parametric & - & 0 & 50.00\% & - & 0.53 \\
0.7 & 1000 & Original & 124.0\% & 50.5\% & 50.00\% & 1991237 & TL \\
\midrule
0.8 & 100 & CORe & 0 & 0 & 50.00\% & 1 & 2.05 \\
0.8 & 100 & parametric & - & 0 & 50.00\% & - & 0.09 \\
0.8 & 100 & Original & 161.9\% & 62.0\% & 50.00\% & 19745207 & TL \\
\midrule
0.8 & 1000 & CORe & 0 & 0 & 50.00\% & 1 & 0.49 \\
0.8 & 1000 & parametric & - & 0 & 50.00\% & - & 0.52 \\
0.8 & 1000 & Original & 197.4\% & 91.6\% & 50.00\% & 1772636 & TL \\
\midrule
0.9 & 100 & CORe & 0 & 0 & 50.00\% & 1 & 0.87 \\
0.9 & 100 & parametric & - & 0 & 50.00\% & - & 0.11 \\
0.9 & 100 & Original & 465.1\% & 159.1\% & 50.00\% & 19616860 & TL \\
\midrule
0.9 & 1000 & CORe & 0 & 0 & 50.00\% & 1 & 0.47 \\
0.9 & 1000 & parametric & - & 0 & 50.00\% & - & 0.52 \\
0.9 & 1000 & Original & 627.1\% & 220.1\% & 49.90\% & 1793454 & TL \\
\bottomrule
\end{tabular}}
\end{table}

Table~\ref{tab:path_adversarial} shows that the adversarial construction can make these instances substantially harder for CORe when the threshold is close to the boundary \(1-\rho\). For \(\tau=0.6\), CORe reaches the time limit for both \(n=100\) and \(n=1000\), exploring about \(1.64\times 10^7\) and \(1.90\times 10^6\) nodes, respectively. In contrast, once \(\tau\) is increased to \(0.7\) or larger, CORe again solves the instances at the root node with zero root gap. Interestingly, the active fraction remains essentially unchanged across these values of \(\tau\), staying near \(50\%\). This indicates that the computational difficulty is not explained solely by the sparsity level of the final solution. Rather, the more direct mechanism is the margin in the CORe regime conditions: when \(\tau\) is close to \(1-\rho=0.55\), many inactive coordinates are nearly on the boundary \(|g_i|=\tau\), making the coordinate-wise certificates much less decisive. The specialized parametric algorithm solves all instances quickly, as expected for path graphs, while the original big-M formulation remains weak and fails to close the gap within the time limit in every setting.

\subsubsection{Tree graphs}

\begin{table}[t]
\centering
\caption{Average computational results for QP with indicators on tree graphs as the problem size $n$ varies.}
\label{tab:tree_problem_size}
\resizebox{\columnwidth}{!}{\begin{tabular}{ccrrrrr}
\toprule
$n$ & Formulation & Root gap & End gap & Active (\%) & Nodes & Time (s) \\
\midrule
100   & CORe       & 0& 0& 15.00\%  & 1        & 0.06 \\
100   & parametric & -& 0& 15.00\%  & -        & 0.10 \\
100   & Original   & 59.9\%& 14.9\%& 15.00\%  & 16235418 & TL   \\
\midrule
200   & CORe       & 0& 0& 12.80\%  & 1        & 0.12 \\
200   & parametric & -& 0& 12.80\%  & -        & 0.15 \\
200   & Original   & 62.5\%& 37.3\%& 12.80\%  & 7666071  & TL   \\
\midrule
300   & CORe       & 0& 0& 12.87\%  & 1        & 0.13 \\
300   & parametric & -& 0& 12.87\%  & -        & 0.21 \\
300   & Original   & 61.4\%& 46.8\%& 12.87\%  & 5665235  & TL   \\
\midrule
500   & CORe       & 0& 0& 11.76\%  & 1        & 0.15 \\
500   & parametric & -& 0& 11.76\%  & -        & 0.35 \\
500   & Original   & 61.5\%& 52.5\%& 11.76\%  & 3628496  & TL   \\
\midrule
1000  & CORe       & 0& 0& 9.60\%  & 1        & 0.28 \\
1000  & parametric & -& 0& 9.60\%  & -        & 0.73 \\
1000  & Original   & 63.0\%& 58.5\%& 9.60\%  & 1846712  & TL   \\
\midrule
2000  & CORe       & 0& 0& 9.55\% & 1        & 0.59 \\
2000  & parametric & -& 0& 9.51\% & -        & 1.52 \\
2000  & Original   & 62.6\%& 60.5\%& 9.51\% & 881081   & TL   \\
\midrule
5000  & CORe       & 0& 0& 8.09\% & 1        & 1.41 \\
5000  & parametric & -& 0& 8.07\% & -        & 3.84 \\
5000  & Original   & 64.8\%& 64.4\%& 8.07\% & 258943   & TL   \\
\midrule
10000 & CORe       & 0& 0& 7.31\% & 1        & 2.70 \\
10000 & parametric & -& 0& 7.30\% & -        & 7.74 \\
10000 & Original   & 64.2\%& 64.0\%& 7.31\% & 117797   & TL   \\
\bottomrule
\end{tabular}}
\end{table}

\begin{figure}[t]
    \centering
    \includegraphics[width=0.72\linewidth]{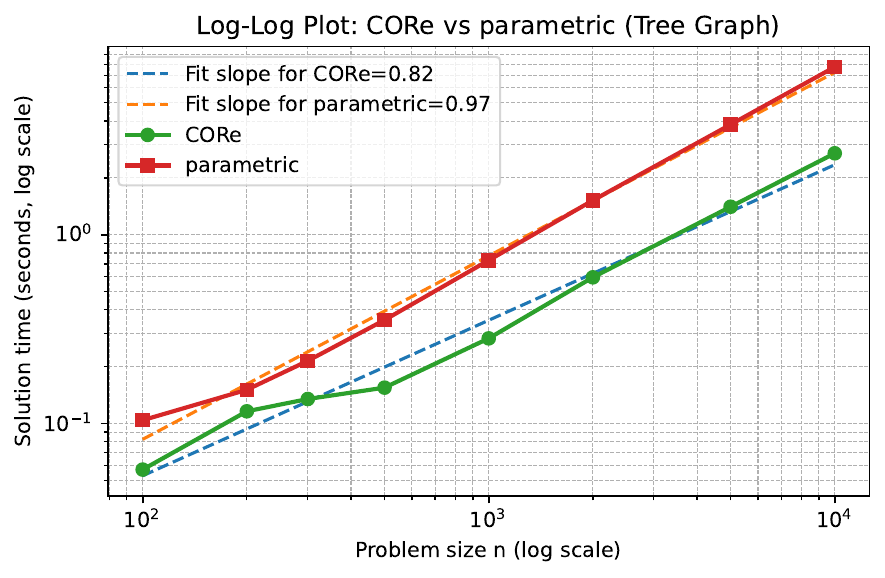}
    \caption{Log-log plot of solution time versus problem size for the CORe formulation and parametric algorithm on tree-graph instances.}
    \label{fig:tree_loglog}
\end{figure}

The data generation process for the tree graphs is the same as in Section \ref{sec:qp_experiments} except that the sparsity pattern of \(Q\) is chosen to be a random tree. For each problem size \(n\), we first sample a uniformly random labeled tree on \(n\) nodes using a Prüfer sequence. Each tree edge \((i,j)\) is assigned an independent weight \(w_{ij}\sim \mathrm{Unif}(0.1,1)\).

As shown in Table~\ref{tab:tree_problem_size}, the three approaches exhibit sharply different behavior on tree graphs. The original big-$M$ formulation fails to solve any of the tested instances within the time limit, and its relaxation remains weak throughout, with root gaps around $60\%$ and large end gaps at termination. In contrast, both CORe and the parametric algorithm solve all instances to proven optimality. Moreover, CORe is consistently the fastest method across all problem sizes, solving essentially at the root node and requiring only a fraction of a second for small and medium instances and only a few seconds even when $n=10{,}000$. The parametric algorithm also scales well and remains exact, but is uniformly slower than CORe by a moderate constant factor. Overall, Table~\ref{tab:tree_problem_size} shows that on tree-structured instances, CORe substantially outperforms the standard formulation and is also competitive with, and in these experiments faster than, the specialized parametric algorithm.

Figure~\ref{fig:tree_loglog} further illustrates the scalability of CORe and the parametric algorithm on tree graphs. Both curves are close to linear on the log-log scale, indicating subquadratic empirical runtime growth as the problem size increases. At the same time, the CORe curve lies consistently below the parametric curve over the entire tested range, confirming its lower runtime across all tested instance sizes. The fitted slopes are approximately $0.82$ for CORe and $0.97$ for the parametric algorithm, suggesting that CORe not only has a smaller runtime constant but also exhibits slightly better empirical scaling in the tested range.

\subsection{Robust Single-Index Models}\label{subsec:exp-single-index}

We demonstrate that the benefits of CORe extend beyond the basic quadratic setting, focusing on the robust classification task via logistic regression
introduced in Section~\ref{subsec:classification}. We construct the test instances from the Breast Cancer Wisconsin diagnostic dataset. For each sample size \(n\in \{50,100,200,500\}\), we use the first $n$ observations and all available covariates. The covariates are standardized separately for each value of $n$ by subtracting the sample mean and dividing by the sample standard deviation. The binary response labels are kept in their original form. We do not synthetically corrupt the labels; instead, the robust formulation determines which observations should be treated as contaminated through the binary variables. Since the purpose of this experiment is to compare formulations rather than tune predictive performance, we set the contamination penalty uniformly to $\lambda = 20\log(n)/n$. This choice yields instances with a modest but nonzero number of selected contaminated observations, avoiding degenerate cases in which nearly all or none of the observations are excluded. For the original formulation, we use \(M_i=1\) in the big-$M$ constraints on the correction variables. 

Table~\ref{tab:single_index_model} reports the computational results, where the CORe formulation yields the reformulation \eqref{eq:robust_logistic_core}. Under a one-hour time limit, the CORe formulation consistently outperforms the original big-$M$ formulation in terms of final solution quality. Although both formulations start with large root gaps, CORe makes substantially more progress during branch-and-bound. The final optimality gaps are much smaller for CORe across all problem sizes; in particular, for $n=100$ and $n=200$, CORe closes the gap to below $0.5\%$, while the original formulation still terminates with gaps above $30\%$. Even for the largest instance, $n=500$, CORe reduces the final gap from $84.8\%$ to $48.6\%$. Interestingly, the number of explored nodes is not uniformly smaller. For $n\le 200$, CORe may even explore more nodes, yet it still produces substantially tighter final bounds. This suggests that, for this class of instances, the main benefit of CORe does not come simply from reducing the search tree size, but rather from more effectively encoding the correct regime logic for each observation, leading to stronger subproblems and improved branch-and-bound performance. Moreover, as expected and noted in Remark~\ref{remark:CF-single-index}, CORe is less effective for single-index models than for quadratic programs with indicators.

\begin{table}[t]
\centering
\caption{Computational results for the robust single-index model comparing the original formulation and the CORe formulation.}
\label{tab:single_index_model}
\small
\resizebox{\columnwidth}{!}{\begin{tabular}{ccrrrrrrr}
\toprule
$n$ & Formulation & Root gap & End UB & End LB & End gap & Active (\%) & Nodes & Time (s) \\
\midrule
50  & Original & 100.0\%& 14.872 & 6.693 & 55.0\%& 10.00\% & 7,562,196 & TL\\
50  & CORe  & 91.9\%& 14.502 & 13.516 & 6.8\%& 6.00\%  & 8,741,463 & TL\\
\midrule
100 & Original & 100.0\%& 8.030  & 5.568 & 30.7\%& 5.00\%  & 4,532,493 & TL\\
100 & CORe  & 90.1\%& 7.733  & 7.707 & 0.3\%& 4.00\%  & 7,054,500 & TL\\
\midrule
200 & Original & 100.0\%& 5.123  & 3.398 & 33.7\%& 2.00\%  & 2,131,796 & TL\\
200 & CORe  & 100.0\%& 4.847  & 4.837 & 0.2\%& 2.00\%  & 3,444,753 & TL\\
\midrule
500 & Original & 100.0\%& 7.615  & 1.160 & 84.8\%& 3.20\% & 1,052,618 & TL\\
500 & CORe  & 100.0\%& 6.600  & 3.394 & 48.6\%& 1.80\%  & 525,868   & TL\\
\bottomrule
\end{tabular}}
\end{table}

\subsection{Real-world Accelerometer Data}

We next evaluate CORe on a real-world accelerometer dataset collected from a wearable sensor worn by a participant across a variety of daily activities \cite{casale2011human,casale2012personalization}. The dataset consists of a sequence of acceleration measurements recorded over time. Intuitively, the signal reflects the intensity of the participant's body motion: periods of little movement tend to produce relatively stable acceleration values, whereas activities such as walking, climbing stairs, or other physical motions generate larger and more rapidly changing measurements. As a result, the time series exhibits alternating periods of low and high activity, along with substantial measurement noise and occasional abrupt fluctuations. Such data naturally arise in activity recognition and human motion monitoring applications, where the objective is to recover an underlying activity signal from noisy sensor observations. The temporal nature of the measurements induces a path structure, as neighboring time points are expected to exhibit similar activity levels, making this dataset a natural real-world test case for sparse and smooth signal estimation models.

This patient was “working at a computer” until timestamp 4,415; then engaged in “standing up, walking, and going upstairs” until timestamp 4,735; followed by “standing” from timestamp 4,735 to 5,854, from 8,072 to 9,044, and again from 9,045 to 9,720. Subsequently, they were “walking” from timestamp 5,854 to 8,072; involved in “going up or down stairs” from timestamp 9,044 to 9,435; “walking and talking with someone” from timestamp 9,720 to 10,430; and “talking while standing” from timestamp 10,457 to 13,800 (with the status between timestamps 10,430 and 10,457 being unknown).

We partition the accelerometer signal into non-overlapping blocks of $K$ consecutive observations and introduce one latent state \(x_t\) for each block. Thus, \(t=1,\ldots,T\) indexes a coarser time scale, where $T$ is the number of blocks. This is natural because physical activities typically persist across multiple sensor readings, whereas individual measurements can be noisy. Following the preprocessing in \cite{bhathena2025parametric}, we set \(K=10\), so each latent state represents ten consecutive accelerometer measurements. We can model the problem as the following problem:

\begin{equation}
\label{eq:accelerometer_nonrobust}
\begin{aligned}
\min_{x,z}\quad
& \sum_{t=1}^{T}\sum_{k=1}^{K}
\left(y_{k,t}-x_t\right)^2
+ \frac{1}{2}x_1^2 + \frac{1}{2} \sum_{t=2}^{T}\left(x_t-x_{t-1}\right)^2
+ \lambda \sum_{t=1}^{T}z_t \\
\text{s.t.}\quad
& x_t(1-z_t)=0,
\qquad t=1,\ldots,T,\\
& z_t\in\{0,1\},
\qquad t=1,\ldots,T.
\end{aligned}
\end{equation}

\begin{figure}[t]
\centering
\includegraphics[width=0.9\columnwidth]{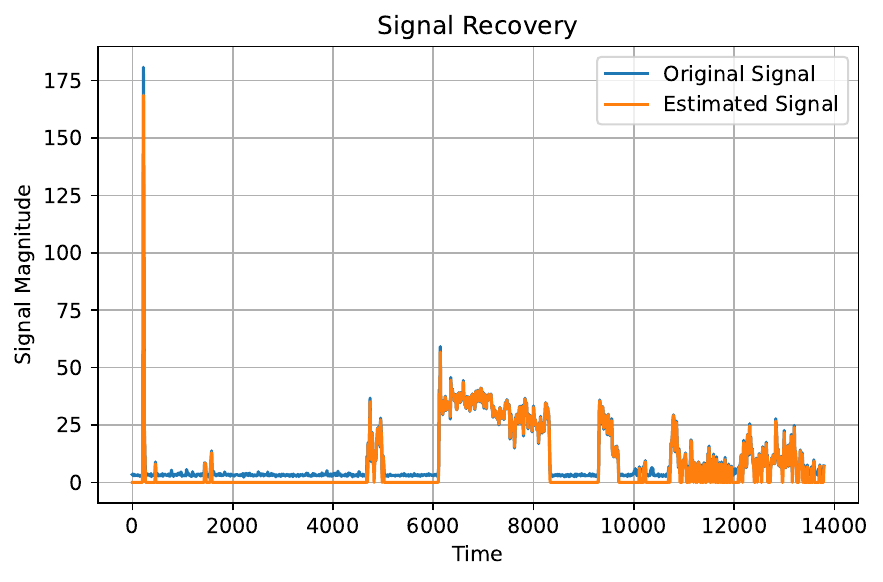}
\caption{Signal recovery on the real-world accelerometer dataset.}
\label{fig:accelerometer_nonrobust}
\end{figure}

Figure~\ref{fig:accelerometer_nonrobust} shows the recovered signal for the accelerometer instance. The estimate produced by CORe captures the main temporal pattern while filtering out near-zero and high-frequency variation. This instance is solved to proven optimality in \(0.1\) seconds using CORe, compared with \(0.7\) seconds using the parametric algorithm of \cite{bhathena2025parametric}. This demonstrates that CORe can efficiently recover a sparse and smooth latent signal from real sensor data.

\section{Conclusion}
\label{sec:conclusion}
We introduced the Coordinate Optimality Reformulation (CORe) framework for
mixed-integer convex programs with indicator variables. The main idea is to exploit the coordinate-wise analytic structure already present in the continuous subproblem. By comparing the inactive and active regimes of each coordinate, CORe derives optimality conditions, closed-form coordinate minimizers, and threshold rules that can be embedded directly into the mixed-integer formulation. These constraints remove regime assignments that cannot occur at global optimality, while preserving
all globally optimal solutions.

We showed the broad applicability of this framework on indicator-constrained models. For quadratic programs with indicators, CORe yields explicit three-regime disjunctions based on coordinate-wise optimality. The same logic also extends to structured settings, such as quadratic programs with unbalanced bipartite structure, and to nonlinear models, such as robust single-index formulations. When suitable bounds are available, these disjunctions can be encoded through strong extended formulations, thereby exposing structure that is absent from standard big-$M$ formulations.

The computational results demonstrate that this additional structure can lead to substantial improvements in branch-and-bound performance. Across the tested quadratic instances, CORe consistently outperforms the original big-$M$ formulation, often solving instances at or near the root node. On path- and tree-structured instances, CORe is also competitive with specialized parametric algorithms and is faster on the tested instances. The robust single-index experiments further show that the benefits of CORe are not limited to quadratic objectives: even when the root gaps
remain large, the CORe formulation produces stronger subproblems and substantially better final bounds.

Several directions remain open for future work. One direction is to characterize more general classes of mixed-integer convex problems that admit useful coordinate-wise optimality certificates. Another is to extend the framework to settings with additional constraints that couple the indicator-controlled coordinates, where coordinate-wise arguments must be combined with dual or sensitivity information. Finally, it would be
interesting to study branching rules and presolve techniques that are designed specifically to exploit the regime structure induced by CORe.

\section*{Declarations}
\textbf{Conflicts of interest} The authors declare that they have no conflict of interest.

\section*{Data Availability}
The datasets generated and analyzed in this study are available at:\\
\url{https://github.com/AtomXT/QP_indicator}.

\section*{Code Availability}
The source code and documentation used for the computational experiments are available at:\\
\url{https://github.com/AtomXT/QP_indicator}.

\bibliographystyle{plain}
\bibliography{reference}

\appendix

\section{Proof of Proposition~\ref{prop:path_counterexample_exp}}
\label{sec:proof_counter}

\begin{proof}

Let \(Q\in\mathbb R^{n\times n}\) be the tridiagonal matrix
\[
Q_{ii}=1 \quad (i=1,\dots,n),
\qquad
Q_{i,i+1}=Q_{i+1,i}=\frac12 \quad (i=1,\dots,n-1),
\]
and all other entries equal to zero. Let
\[
c_i=-1,
\qquad
\tau_i=\tau \in \left(\frac12,1\right),
\]
and choose \(\lambda_i=\tau^2/2\) for all \(i=1,\dots,n\). Assume \(M\ge 1\). Then the number of feasible integer leaves of the CORe formulation is at least the number of maximal independent sets of the path graph \(P_n\), and hence grows exponentially in \(n\).

We divide the proof into three steps.

\medskip
\noindent
\textbf{Step 1: The instance is valid and has path sparsity.}
By construction, the sparsity graph of \(Q\) is the path graph \(P_n\). Moreover, \(Q\) is positive definite: its eigenvalues are
\[
\lambda_k(Q)=1+\cos\!\Bigl(\frac{k\pi}{n+1}\Bigr),
\qquad k=1,\dots,n,
\]
and therefore \(\lambda_k(Q)>0\) for all \(k\). Hence \(Q\succ 0\) and \(Q_{ii}>0\) for every \(i\).

For this instance, 
\[
g_i
=
c_i+\sum_{j\neq i}Q_{ij}x_j
=
-1+\frac12 x_{i-1}+\frac12 x_{i+1}.
\]
Since \(Q_{ii}=1\), the three CORe regimes are
\begin{align*}
\mathcal D_{i, \mathrm{CF}}^0 &:= \{x_i=0,\; |g_i|\le \tau\},\\
\mathcal D_{i, \mathrm{CF}}^+ &:= \{x_i=-g_i,\; g_i\ge \tau\},\\
\mathcal D_{i, \mathrm{CF}}^- &:= \{x_i=-g_i,\; g_i\le -\tau\}.
\end{align*}

\medskip
\noindent
\textbf{Step 2: Every maximal independent set gives a feasible integer leaf.}
Let \(S\subseteq \{1,\dots,n\}\) be any maximal independent set of the path graph \(P_n\). Define
\[
x_i=
\begin{cases}
1,& i\in S,\\
0,& i\notin S.
\end{cases}
\]
Since \(M_i\ge 1\), this choice satisfies the box constraints \(-M_iz_i\le x_i\le M_iz_i\) once the corresponding binary variables are chosen consistently.

We claim that \(x\) satisfies the CORe disjunction at every coordinate.

\smallskip
\noindent
\emph{Case 1: \(i\in S\).}
Because \(S\) is an independent set, neither neighbor of \(i\) belongs to \(S\). Hence
\[
x_{i-1}=x_{i+1}=0
\]
(with the obvious interpretation at the endpoints), and therefore
\[
g_i=-1.
\]
Thus
\[
x_i=1=-g_i,
\qquad
g_i=-1\le -\tau,
\]
because \(\tau<1\). Therefore \(i\) satisfies \(\mathcal D_{i, \mathrm{CF}}^-\).

\smallskip
\noindent
\emph{Case 2: \(i\notin S\).}
Since \(S\) is maximal, vertex \(i\) must have at least one neighbor in \(S\); otherwise \(i\) could be added to \(S\), contradicting maximality. On a path, \(i\) has at most two neighbors, and because \(S\) is independent, both neighbors may belong to \(S\) only if they are \(i-1\) and \(i+1\). Therefore the number of neighbors of \(i\) that belong to \(S\) is either \(1\) or \(2\). Consequently,
\[
g_i=-1+\frac12 x_{i-1}+\frac12 x_{i+1}
\in \left\{-\frac12,\,0\right\}.
\]
Since \(\tau>\frac12\), it follows that
\[
|g_i|\le \tau.
\]
Because \(x_i=0\), we conclude that \(i\) satisfies \(\mathcal D_{i, \mathrm{CF}}^0\).

\smallskip
Thus, every coordinate satisfies one of the three CORe regimes, so the vector \(x\) is consistent with a complete regime assignment. Hence, every maximal independent set \(S\) yields a feasible integer leaf of the CORe formulation.

\medskip
\noindent
\textbf{Step 3: The number of maximal independent sets of \(P_n\) is exponential.}

Even for path graphs, the number of maximal independent sets grows exponentially: if $a_n$ denotes the number of maximal independent sets of the path $P_n$, then $a_n=a_{n-2}+a_{n-3}$, and hence $a_n=\Theta(\rho^n)$, where
$\rho\approx 1.3247$ is the real root of $\rho^3=\rho+1$
\cite{wilf1986number,furedi1987number}.

Since every maximal independent set yields a feasible integer leaf, the number of feasible integer leaves of the CORe formulation is at least \(a_n\), and is therefore exponential in \(n\). \qed
\end{proof}

\section{Exact Polynomial-Time Algorithm for Quadratic Programs with Unbalanced Bipartite Structure via Geometric Pruning}
\label{sec:poly-algo}
In this section, we present a polynomial-time algorithm for solving Problem~\eqref{problem:low_rank_general} when $m$ is fixed and $G\succ 0$. Without loss of generality, we take $D=I_n$, $G=I_m$, and $b=0$: substituting $\tilde x = D^{1/2}x$ and $\tilde y = G^{1/2}\bigl(y+G^{-1}b\bigr)$ preserves the indicator constraints (up to rescaling the constants $M_i$) and puts the objective in this normalized form, with coupling matrix $V := D^{-1/2}FG^{-1/2}$ and a suitably updated linear coefficient vector $c$. Positive semidefiniteness of the original quadratic form is then equivalent to $I - VV^\top \succeq 0$.

\subsection{Rank-2 case}
To explicitly illustrate the geometric decomposition and candidate enumeration, we first consider the case $m=2$. After the normalization above, and writing the coupling matrix columnwise as $V=[\,u\;\; v\,]$ with $u, v\in \mathbb{R}^n$ such that $I-uu^\top - vv^\top\succ 0$, Problem~\eqref{problem:low_rank_general} reads

\begin{subequations}\label{eq: MIQP_rank2_ext}
	\begin{align}
		\min_{x\in\mathbb{R}^n,z\in\{0,1\}^n, y_1, y_2\in \mathbb{R}}\qquad& \dfrac{1}{2}y_1^2+\dfrac{1}{2}y_2^2 +y_1(u^\top x)+y_2(v^\top x) + \dfrac{1}{2}x^\top x+c^\top x+\lambda^\top z\label{eq: MIQP obj_rank2_ext}\\
		\text{s.t.}\qquad &x_i(1-z_i)=0,\quad   i=1,2\ldots n,\label{eq: MIQP_const_rank2_ext}
	\end{align}
\end{subequations}
where we write the indicator constraints in complementarity form (equivalent to the big-$M$ form for sufficiently large $M_i$) and $c$ denotes the normalized linear coefficient vector.

\begin{theorem}
    The optimization problem~\eqref{eq: MIQP_rank2_ext} can be solved in $O(n^{3})$ time and memory using Algorithm~\ref{alg: rank-2}.
\end{theorem}

\begin{proof}
Projecting out $x$ and $z$, problem~\eqref{eq: MIQP_rank2_ext} can be rewritten as

\begin{align}
    &\min_{y_1,y_2} \dfrac{1}{2}y_1^2+\dfrac{1}{2}y_2^2 + \sum_{i=1}^n \underbrace{\min_{x_i(1-z_i)=0} 1/2x_i^2+ (u_iy_1+v_iy_2+c_i)x_i+\lambda_i z_i}_{:=g_i(y_1,y_2)} \notag\\
    =&\min_{y_1,y_2} \dfrac{1}{2}y_1^2+\dfrac{1}{2}y_2^2 + \sum_{i=1}^n g_i(y_1,y_2),\label{eq_g}
\end{align}
where
$$
g_i(y_1,y_2) = \begin{cases}
    0 & -\sqrt{2\lambda_i}\leq u_iy_1+v_iy_2+c_i\leq \sqrt{2\lambda_i}\\
    -\frac{(u_iy_1+v_iy_2+c_i)^2}{2}+\lambda_i & \text{otherwise}.
\end{cases}
$$
Solving~\eqref{eq_g} boils down to finding the pieces of $\{g_i(y_1,y_2)\}$ that contain the optimal solution $(y_1^*,y_2^*)$. 
At first glance, this appears to have exponential complexity: each $g_i(y_1,y_2)$ has two pieces, and to find the pieces that contain the optimal solution, one may need to exhaust all $2^n$ possible combinations of these pieces. However, we show that the majority of these combinations are infeasible for~\eqref{eq_g} and hence can be ``pruned". 

To see this, note that each function $g_i$ divides the $\mathbb{R}^2$-space in two regions with its two supporting lines: for all points in the set $\gC^1_i = \{(y_1,y_2): -\sqrt{2\lambda_i}\leq u_iy_1+v_iy_2+c_i\leq \sqrt{2\lambda_i}\}$, we have $z_i=0$ and $g_i(y_1,y_2)=0$. Similarly, for all points in the set $\gC_i^2 = \mathbb{R}^2\backslash \gC_i^1$, we have $z_i=1$ and $g_i(y_1,y_2) = -\frac{(u_iy_1+v_iy_2+c_i)^2}{2}+\lambda_i$. Consider the set of supporting lines $\gL = \bigcup_{i=1}^n\left\{\ell_{i,1}(y_1,y_2)=0, \ell_{i,2}(y_1,y_2)=0\right\}$, where $\ell_{i,1}(y_1,y_2) = u_iy_1+v_iy_2+c_i - \sqrt{2\lambda_i}$ and $\ell_{i,2}(y_1,y_2) = u_iy_1+v_iy_2+c_i + \sqrt{2\lambda_i}$. 

This decomposes the $\mathbb{R}^2$-space into disjoint polytopes. Within each polytope, the sparsity pattern of $\{z_i\}$ remains the same. Therefore, by counting the number of polytopes, we can characterize the number of feasible sparsity patterns for~\eqref{eq_g}.

\begin{algorithm}[ht]
	\caption{Efficient algorithm for the case $m=2$}\label{alg: rank-2}
	\textbf{Input:} Parameters of the optimization problem~\eqref{eq: MIQP_rank2_ext}\\
	\textbf{Output:} Optimal solution to~\eqref{eq: MIQP_rank2_ext}
	\begin{algorithmic}[1]
        \State Set $\gZ = \{\}$.
        \State Compute the set of supporting lines $\gL$ from~\eqref{eq_g}. \Comment{It can be done in $O(n)$ time}
		\For{$\ell_i = u_iy_1+v_iy_2+c_i\pm \sqrt{2\lambda_i}\in \gL$}
            \State Find the lower and upper bounds using Equation~\eqref{eq_intervals}.
            \State Sort these bounds and obtain the corresponding partial sparsity patterns $\{\bar z^{(k)}\}_{k=1}^{2n-1}\subset \mathbb{R}^{n-1}$.
            \State Obtain the full candidate sparsity patterns $\{z^{(2k-1)}, z^{(2k)}\}_{k=1}^{2n-1}$, as $z^{(2k-1)}_j=z^{(2k)}_j = \bar z^{(k)}_j$ for $j\not=i$, and $z^{(2k-1)}_i=0$, $z^{(2k)}_i=1$.
            \State Set $\gZ\leftarrow \gZ\cup \{z^{(2k-1)}, z^{(2k)}\}_{k=1}^{2n-1}$
		\EndFor
  \State $f^* = +\infty$, $x^* = \{\}$
  \For{$\bar z\in \gZ$}
  \State find the optimal solution and optimal cost $(\bar x, \bar f)$ of~\eqref{eq: MIQP_rank2_ext} subject to $z = \bar z$
  \If{$\bar f<f^*$}
  \State Set $f^*\leftarrow \bar f$ and $x^*\leftarrow \bar x$.
  \EndIf
  \EndFor
  \State Return {$(x^*, f^*)$}
	\end{algorithmic}
\end{algorithm}

\begin{lemma}
    The set of lines in $\gL$ decomposes $\mathbb{R}^2$-space into at most $\max\{n(4n-2),3\}$ polytopes.
\end{lemma}
\begin{proof}
     Each line in $\ell_i\in \gL$ can be the face of at most $4n-2$ polytopes. To see this, note that the set of lines $\gL\backslash \ell_i$ intersects $\ell_i$ in at most $2n-2$ points. This creates at most $2n-1$ segments on $\ell_i$, each of which can be a face of at most two polytopes. 
     Since each polytope has at least one face in the set $\gL$
    and the size of $\gL$ is $2n$, the total number of polytopes is upper bounded by ${2n(4n-2)}$.

    Next, we show that this bound can be improved by a factor of 2. We consider two cases. First, suppose that there exists a polytope that has only one face in $\gL$. This implies that the lines in $\gL$ do not intersect, i.e., they are parallel. In such a scenario, since we have $2n$ parallel lines, the number of polytopes is upper bounded by $2n+1$. Second, suppose that all polytopes have at least two faces in $\gL$. Based on our previous argument, the total number of polytopes in this case is upper bounded by $\frac{2n(4n-2)}{2} = n(4n-2)$. Combining these two cases leads $\max\{n(4n-2),2n+1\} = \max\{n(4n-2),3\}$ polytopes. \qed
\end{proof}

Next, we explain our algorithm. To find an optimal solution, it is sufficient to first determine the sparsity pattern associated with each polytope defined by the lines in 
$\gL$. By Lemma 1, the number of such polytopes is bounded by $O(n^2)$, leading to at most $O(n^2)$ distinct sparsity patterns. Once these patterns are identified, an optimal solution can be determined by solving a convex quadratic program constrained by each identified sparsity pattern. 

Indeed, the facets of the polytope are included in the lines in $\gL$. Therefore, to identify candidate sparsity patterns within the polytopes, it suffices to determine them at these facets. For instance, to identify the sparsity pattern for all polytopes that are adjacent to line $\ell_i$ (equivalently, $\ell_i$ defines a facet), we consider all pairs $\gY_i = \{(y_1,y_2): \ell_i(y_1,y_2) = u_iy_1+v_iy_2 +c_i\pm\sqrt{2\lambda_i}= 0\}$. Without loss of generality, let us assume that $u_i\not=0$. Then, this implies that $y_1 = -(v_i/u_i)y_2-(c_i\pm\sqrt{2\lambda_i})/a_i$. For every $(y_1,y_2)\in \gY_i$, $z_j = 0, j\not=i$ if and only if 
\begin{align}
    &-\sqrt{2\lambda_j}\leq u_jy_1+v_jy_2+c_j\leq \sqrt{2\lambda_j} \notag\\
    \iff& -\sqrt{2\lambda_j}\leq \left(v_j-\frac{u_jv_i}{u_i}\right)y_2+c_j-u_j(c_i\pm\sqrt{2\lambda_i})\leq \sqrt{2\lambda_j}\notag\\
    \iff& \underbrace{\frac{-\sqrt{2\lambda_j}+u_j(c_i\pm\sqrt{2\lambda_i})-c_j}{\left(v_j-\frac{u_jv_i}{u_i}\right)}}_{\alpha_i}\leq y_2\leq \underbrace{\frac{\sqrt{2\lambda_j}+u_j(c_i\pm\sqrt{2\lambda_i})-c_j}{\left(v_j-\frac{u_jv_i}{u_i}\right)}}_{\beta_i}.\label{eq_intervals}
\end{align}
Combining and sorting the above upper and lower bounds yields $ 2n-1$ intervals for $ y_2$. Within each interval, the sparsity pattern of $\{z_j\}_{j\not=i}$ can be found and remains fixed. Therefore, on the line $\ell_i$, we obtain $2n-1$ sparsity patterns $\{\bar z^{(k)}\}_{k=1}^{2n-1}\subset \mathbb{R}^{n-1}$, which are the possible sparsity patterns for the variables $\{z_j\}_{j\not=i}$ on the polytopes adjacent to $\ell_i$. To find the full sparsity pattern for these polytopes, it remains to determine $z_i$, which can be either $z_i=0$ or $z_i=1$, depending on which side of $\ell_i$ the polytope resides in. This implies that the set of candidate sparsity patterns for the polytopes adjacent to $\ell_i$ can be defined as $\{z^{(2k-1)}, z^{(2k)}\}_{k=1}^{2n-1}$, where both $z^{(2k-1)}$ and $z^{(2k)}$ share the same sparsity pattern at coordinates $j\not=i$, and only differ at coordinate $i$. The runtime for finding these candidate sparsity patterns is $O(n\log n)$ (due to the sorting step).

Performing the above steps for all $\ell\in \gL$ yields all $O(n^2)$ candidate sparsity patterns. These sparsity patterns can be found in $O(n^2\log n)$ time. Moreover, for each fixed sparsity pattern, the optimal solution can be found in $O(n)$ by solving a convex quadratic problem, leading to $O(n^3)$ total complexity. \qed
\end{proof}

\subsection{Generalization to General \texorpdfstring{$m$}{m}}
\label{subsec:poly-algo-rankm}

The geometric pruning framework described for $m=2$ naturally extends to general fixed $m$. After the normalization above, with coupling matrix $V\in \mathbb{R}^{n\times m}$ satisfying $I-VV^\top\succ 0$, Problem~\eqref{problem:low_rank_general} reads
\[
\min_{x\in\mathbb{R}^n, z\in\{0,1\}^n, y\in\mathbb{R}^m} \frac{1}{2}\|y\|^2 + y^\top V^\top x + \frac{1}{2}x^\top x + c^\top x + \lambda^\top z, \quad x_i(1-z_i)=0.
\]

For fixed $y$, the problem separates across coordinates, giving piecewise functions
\[
g_i(y) = 
\begin{cases}
0 & \text{if } |V_i^\top y + c_i| \le \sqrt{2 \lambda_i},\\
-\frac{(V_i^\top y + c_i)^2}{2} + \lambda_i & \text{otherwise,}
\end{cases}
\]
where $V_i^\top$ denotes the $i$-th row of $V$.  
Each $g_i$ induces two parallel hyperplanes in $\mathbb{R}^m$,
\[
V_i^\top y + c_i = \pm \sqrt{2 \lambda_i}.
\]

The union of these $2n$ hyperplanes decomposes $\mathbb{R}^m$ into polytopes, inside each of which the sparsity pattern $z$ is fixed.  
To enumerate candidate sparsity patterns efficiently, each hyperplane can be treated as a facet, and one coordinate of $y$ can be parametrized in terms of the remaining $m-1$ coordinates.  
For instance, assume without loss of generality that $V_{i,1}\neq 0$, we can write
\[
y_1 = - \frac{1}{V_{i,1}} \Big( V_{i,2:m}^\top y_{2:m} + c_i \pm \sqrt{2 \lambda_i} \Big).
\]

For any $j\neq i$, the condition $z_j=0$ is equivalent to
\[
-\sqrt{2\lambda_j} \le V_j^\top y + c_j \le \sqrt{2\lambda_j}.
\]
Substituting the above expression for $y_1$ yields
\[
-\sqrt{2\lambda_j}
\le
\Big( V_{j,2:m} - \frac{V_{j,1}}{V_{i,1}} V_{i,2:m} \Big)^\top y_{2:m}
+ c_j - \frac{V_{j,1}}{V_{i,1}} (c_i \pm \sqrt{2\lambda_i})
\le
\sqrt{2\lambda_j}.
\]
Fixing all but one coordinate of $y_{2:m}$ (e.g., $y_3,\dots,y_m$), this inequality reduces to an interval constraint on the remaining coordinate, say $y_2$,
\[
\alpha_{ij}^{\pm} \le y_2 \le \beta_{ij}^{\pm},
\]
where $\alpha_{ij}^{\pm}$ and $\beta_{ij}^{\pm}$ are obtained by dividing the above bounds by the coefficient of $y_2$.

Sorting these interval boundaries partitions the remaining coordinates into regions with fixed partial sparsity patterns. Combining with the choice of $z_i$ on each side of the hyperplane yields all candidate sparsity patterns adjacent to this facet. Repeating this for all hyperplanes generates all $O(n^m)$ candidate sparsity patterns.

\begin{lemma}
\label{lem:polytopes-rankm}
The $2n$ hyperplanes in $\mathbb{R}^m$ decompose the space into at most $O(n^m)$ polytopes, each corresponding to a distinct candidate sparsity pattern for $z$.
\end{lemma}

Once the candidate sparsity patterns are identified, the optimal solution can be found by solving a convex quadratic program for each pattern, leading to a total complexity of $O(n^{m+1})$.  
This generalizes the rank-2 algorithm while preserving the same geometric pruning principle.

\begin{corollary}
\label{cor:rankm}
For any fixed $m$, Problem~\eqref{problem:low_rank_general} with $G\succ 0$ can be solved in $O(n^{m+1})$ time and memory using the geometric pruning algorithm described above.
\end{corollary}

\begin{proof}
By Lemma~\ref{lem:polytopes-rankm}, the number of candidate sparsity patterns is at most $O(n^m)$. For each fixed pattern, the convex quadratic problem in $x$ can be solved in $O(n)$ time. Therefore, enumerating all patterns and solving the corresponding convex optimization problems yields a total complexity of $O(n^{m+1})$. \qed
\end{proof}

\end{document}

%% file: math_commands.tex
\usepackage{amsmath,amsfonts,bm}

\newcommand*{\rom}[1]{%
\textup{\uppercase\expandafter{\romannumeral#1}}%
}

\def\1{\bm{1}}

\DeclareMathAlphabet{\mathsfit}{\encodingdefault}{\sfdefault}{m}{sl}
\SetMathAlphabet{\mathsfit}{bold}{\encodingdefault}{\sfdefault}{bx}{n}

\def\gC{{\mathcal{C}}}

\def\gL{{\mathcal{L}}}

\def\gY{{\mathcal{Y}}}
\def\gZ{{\mathcal{Z}}}

\newcommand{\ignore}[1]{}

\newtheorem{assumption}{Assumption}